\documentclass[11pt,letterpaper]{amsart}
\usepackage{amsfonts, amsmath, amssymb, amscd, amsthm, color, graphicx, mathrsfs, wasysym, setspace, mdwlist, calc, float, tikz-cd}
\usepackage{setspace}
\usepackage{hyperref}
\usepackage{enumitem}

\usepackage{xcolor}

\usepackage{hyperref}
\hypersetup{linktocpage}

\hypersetup{colorlinks,
    linkcolor={red!50!black},
    citecolor={blue!80!black},
    urlcolor={blue!80!black}}
\usepackage{float}

\newcommand{\e}{\varepsilon}
\newcommand{\NN}{\mathbf{N}}
\newcommand{\QQ}{\mathbf{Q}}
\newcommand{\ZZ}{\mathbf{Z}}
\newcommand{\Z}{\mathbf{Z}}
\newcommand{\N}{\mathbf{N}}
\renewcommand{\phi}{\varphi}
\renewcommand{\d}{\mathrm{d}}
\renewcommand{\ll }{\langle\hspace{-.7mm}\langle }
\newcommand{\rr }{\rangle\hspace{-.7mm}\rangle }
\newcommand{\lab}{{\bf lab}}

\newcommand{\CGXH}{\Gamma (G,X\sqcup \mathcal H)}
\newcommand{\dl}{\widehat\d_i}

\newcommand{\la}{\langle}
\newcommand{\ra}{\rangle}
\newcommand{\Ad}{\mathrm{Ad}}
\newcommand{\Aut}{\mathrm{Aut}}
\newcommand{\Inn}{\mathrm{Inn}}
\newcommand{\Linn}{\mathrm{Linn}}
\newcommand{\QZ}{\mathrm{QZ}}
\newcommand{\Hl}{\{ H_i\}_{i\in I}}
\newcommand{\h}{\hookrightarrow_h}
\newcommand{\dA}{\d_{X\sqcup \mathcal H}}

\newtheorem{thm}{Theorem}[section]
\newtheorem{cor}[thm]{Corollary}
\newtheorem{lem}[thm]{Lemma}
\newtheorem{prop}[thm]{Proposition}

\theoremstyle{definition}
\newtheorem{defn}[thm]{Definition}
\newtheorem{conv}[thm]{Convention}
\newtheorem{notat}[thm]{Notation}

\theoremstyle{remark}
\newtheorem{rem}[thm]{Remark}
\newtheorem{ex}[thm]{Example}

\newfont{\eufm}{eufm10}

\begin{document}

\title{Simple $p$-adic Lie groups with abelian Lie algebras}
\author{P.-E. Caprace}
\address[Pierre-Emmanuel Caprace]{UCLouvain -- Institut de Recherche en Math\'ematiques et Physique, 1348 Louvain-la-Neuve, Belgium.}
\email{pierre-emmanuel.caprace@uclouvain.be}
\author{A. Minasyan}
\address[Ashot Minasyan]{CGTA, School of Mathematical Sciences,
University of Southampton, Highfield, Southampton, SO17~1BJ, United Kingdom.}
\email{aminasyan@gmail.com}
\author{D. Osin}
\address[Denis Osin]{Department of Mathematics, Vanderbilt University, Nashville, TN 37240, U.S.A.}
\email{denis.osin@gmail.com}
\thanks{P.-E. Caprace has been supported in part by the FWO and the F.R.S.-FNRS under the EOS programme (project ID 40007542). D. Osin has been supported by the NSF grant DMS-1853989}  

\date{}

\begin{abstract}
For each prime $p$ and each positive integer $d$, we construct the first examples of second countable, topologically simple $p$-adic Lie groups of dimension~$d$ whose Lie algebras are abelian. This answers several questions of Gl\"ockner and Caprace--Monod. The proof relies on a generalization of small cancellation methods that applies to central extensions of acylindrically hyperbolic groups. 
\end{abstract}
\maketitle

\vspace{-2mm}

\begin{flushright}
\begin{minipage}[t]{0.47\linewidth}\itshape\small
It's the simple things in life that are the most extraordinary \dots\\
\vspace{-2mm}

\hfill\upshape \textemdash Paulo Coelho,  \emph{The Alchemist}, 1988.
\end{minipage}
\end{flushright}

\vspace{3mm}

\section{Introduction}

\subsection{Motivation and main results} To a far extent, the structure theory of real and $p$-adic Lie groups can be developed in a unified manner, see \cite[Ch.~III]{Bourbaki} or \cite{Serre}. Striking results, like the fact that every closed subgroup is a Lie subgroup (see \cite[Ch.~III, \S 8, no. 2, Th.~2]{Bourbaki}), are valid in both categories.  An important role in this theory is played by the \emph{adjoint representation}  
\[\Ad_G \colon G \to \mathrm{GL}(L(G))\]
of a real or $p$-adic Lie group $G$ on its Lie algebra $L(G)$; for every element $g\in G$, the linear transformation  $\Ad_G(g)$ is defined as the tangent map at the neutral element of the inner automorphism of $G$ corresponding to $g$ \cite[Ch.~III, \S 3 no. 12]{Bourbaki}. 

The formula 
\[g \exp(x) g^{-1} = \exp(\Ad_G(g).x),\]
valid for all $g \in G$ and all $x$ in a sufficiently small neighborhood of $0$ in $L(G)$ (see \cite[Ch.~III, \S 4, no. 4, Cor.~3(ii)]{Bourbaki}), ensures that the kernel of the adjoint representation coincides with the subgroup $\QZ(G) \leqslant G$ consisting of elements that centralize some open subgroup of $G$. Following Burger--Mozes \cite{BuMo00},  $\QZ(G)$ is   called the \emph{quasi-center} of $G$. 

In general, the quasi-center can be defined for any topological group and it will be a topologically characteristic subgroup (i.e., a subgroup invariant under all homeomorphic automorphisms), which need not be closed. 
If $G$ is connected, we have $\QZ(G) = {Z}(G)$ because the only open subgroup of $G$ is $G$ itself. If $G$ is a $p$-adic Lie group, the corresponding fact only holds locally; that is, $\QZ(G)$ is a closed subgroup whose Lie algebra is the center of $L(G)$. In this case, the closedness of $\QZ(G)$  follows from the fact that it is the kernel of the adjoint representation, which is a continuous homomorphism (see  \cite[Ch.~III, \S 3 no. 12, Proposition~45]{Bourbaki}). 

Just like in the real case, classifying simple groups is of fundamental importance for developing the general structure theory of $p$-adic Lie groups. While the complete classification of simple real Lie groups, due to Killing and Cartan, has been known since the end of the 19th century, the $p$-adic case is far from being understood. In the latter context, it is natural to consider {\it topologically simple groups}, i.e., non-trivial groups that do not possess any nontrivial proper closed normal subgroups. 

It is easy to see that every topologically simple $p$-adic Lie group $G$ satisfies the following dichotomy: {\it either the adjoint representation of $G$ is faithful (in particular, $G$ is linear) or $G= \QZ(G)$ (equivalently, $L(G)$ is abelian)}. 
The topologically simple groups whose adjoint representation is faithful can be characterized by the following statement, which summarizes  several known facts. 

\begin{thm}\label{thm:algebraic}
Let $p$ be a prime. For a topologically simple $p$-adic Lie group $G$ of positive dimension, the following assertions are equivalent:
\begin{enumerate}[label=(\roman*)]
    \item $G$ is compactly generated. 
    \item The adjoint representation  of $G$ is faithful. 
    \item $\QZ(G)=\{1\}$ (equivalently, ${Z}(L(G))=\{0\}$).
    \item $G$ is continuously linear over $\QQ_p$. 
    \item $G$ is algebraic over $\QQ_p$ (more precisely, $G$ is isomorphic to the group of rational points of a $\QQ_p$-simple algebraic $\QQ_p$-group   divided by its center).
\end{enumerate}
\end{thm}
The implication from (i) to (ii) follows from  the main result of \cite{Wil07} (see also \cite[Theorem~A]{CRW} for a more general statement). The fact that (iv) implies (v) is recorded in \cite[Proposition~6.5]{CCLTV}, while the implication from (v) to (i) follows from a classical result of Borel--Tits (see \cite[Theorem~13.4]{BoTi65}). 

However,  whether a non-discrete simple $p$-adic Lie group could fail to satisfy the equivalent conditions of Theorem~\ref{thm:algebraic}  remained unknown. 
The question of the existence of a topologically simple $p$-adic Lie group (of dimension at least $1$) with an abelian Lie algebra explicitly appears in  \cite[Problem 20.2.1]{CM18}. 
Our main result provides the affirmative answer in the general case. We state it in a simplified form here and refer to Theorem~\ref{thm:technical-version} for more detail.

\begin{thm}\label{thm:main}
For each prime $p$ and every positive integer $d$,  there exists a continuum of abstract isomorphism classes of second countable, topologically simple $p$-adic Lie groups of dimension $d$ whose Lie algebras are abelian.
\end{thm}

In the course of proving the claim about cardinality, we show that every countable group can be embedded (as a discrete subgroup) in a $p$-adic Lie group as in Theorem~\ref{thm:main}. This is quite surprising as all of the previously known topologically simple $p$-adic Lie groups were linear. Furthermore, our construction yields examples of topologically simple, totally disconnected, locally compact, second countable groups with compact open subgroups isomorphic to direct products of groups $\ZZ_p$ over different primes (see Theorem~\ref{thm:technical-version}). 

Following Gl\"ockner, a $p$-adic Lie group $G$ is called 
 \emph{extraordinary} if it enjoys all of the following properties (see \cite[Definition~1.4]{Glo17}):
 \begin{itemize}
     \item[(a)] every closed subnormal subgroup of $G$ is open or discrete,
     \item[(b)]  the adjoint action of $G$ is trivial (i.e., $G = \QZ(G)$),
     \item[(c)] the derived subgroup $[H, H]$ of each open subnormal subgroup $H \leqslant G$ is non-discrete.
 \end{itemize}  

While the class of extraordinary groups naturally pops up in the global structure theory of $p$-adic Lie groups (see \cite[Theorem~1.5]{Glo17}), the question of their very existence remained open (see 
\cite[Remark~4.2]{Glo17}). In Section 4, we obtain the affirmative answer as an immediate consequence of Theorem~\ref{thm:main}. 

\begin{cor}\label{cor:extraordinary}
    Extraordinary $p$-adic Lie groups of dimension~$d$ do exist for every $d \geq 1$ and every prime $p$. 
\end{cor}

Motivated by the study of the structure of $\ker\Ad$, Gl\"ockner also asked whether every $p$-adic Lie group $G$ contains closed subgroups $Z\triangleleft K\triangleleft G$ such that $Z$ is abelian, $K/Z$ is discrete and $G/K$ has a continuous injective homomorphism into $\mathrm{GL}_m(\QQ_p)$, for some $m \geq 1$ (see \cite[Problem~1]{Glo16}). We will see in Theorem~\ref{thm:technical-version} (iv) below that every continuous homomorphism from the simple $p$-adic Lie groups from Theorem~\ref{thm:main} to $\mathrm{GL}_m(\QQ_p)$ has trivial image. This implies that the answer to  \cite[Problem~1]{Glo16} is negative. 

Theorem~\ref{thm:main} and Corollary~\ref{cor:extraordinary}  give examples of $p$-adic Lie groups that are very far from linear algebraic groups, and thereby provide relevant illustrations for the question addressed by Benoist--Quint in \cite{BQ}, that builds upon the notions of \emph{regularity} and \emph{weak regularity} introduced by M.~Ratner \cite{Rat95}.

\subsection{Comments on the proof} Perhaps surprisingly, the proof of our main result employs geometric tools from the theory of group actions on hyperbolic spaces. More precisely, the groups from Theorem~\ref{thm:main} arise as closed subgroups of the automorphism group of some  countable discrete groups $G_\infty$, themselves defined as unions of ascending chains $(G_n)$ of finitely generated groups with suitable properties (see Proposition~\ref{prop:asc-chain} below). The  construction of the groups $G_n$ makes use of a generalization of the small cancellation technique developed in \cite{Hull, Osi10, Ols93}, which seems to be of independent interest. 

Recall that the classical theory of small cancellation studies quotient groups of the form $F(X)/\ll \mathcal R\rr$, where $F(X)$ is the free group with a basis $X$ and $\mathcal R$ is a collection of reduced words in the alphabet $X\cup X^{-1}$ that have ``small overlaps". A generalization of this theory to quotients of hyperbolic groups was proposed by Gromov in \cite{Gro} and elaborated by Olshanskii in \cite{Ols93}. 

Although the paper \cite{Ols93} specifically deals with hyperbolic groups, many of the results obtained there apply to relatively hyperbolic groups and, more generally, to acylindrically hyperbolic groups with slight modifications, as demonstrated in \cite{Hull,Osi10}. In this context, a significant amount of work is required to construct relations of the desired form that satisfy the necessary small cancellation conditions. The main novelty of the approach taken in \cite{Hull, Osi10} is the explicit use of hyperbolically embedded subgroups for this purpose. We further generalize this method to the class of the so-called weakly relatively hyperbolic groups and apply it to central extensions of acylindrically hyperbolic groups (see Propositions \ref{Prop:SCW} and \ref{Lem:Zquot}). We expect these results to have other applications to the study of groups acting on hyperbolic spaces.

\begin{rem} An approach combining small cancellation arguments and central extensions was also used in \cite{HMOO} to give examples of $1$-dimensional $p$-adic Lie groups with counter-intuitive properties. However, the groups constructed in \cite{HMOO} are not topologically simple.  It is also worth noting that the proof of the main theorem in  \cite{HMOO} relies on a still unproved result announced in \cite{IO}. 
\end{rem}

\subsection{Organization of the paper} In the next section, we recall the necessary background from the theory of groups acting on hyperbolic spaces. Section~\ref{sec:small-cancellation} is devoted to the generalization of the small cancellation technique discussed above. The main result of that section is Proposition~\ref{Lem:Zquot}. Those readers who are not interested in small cancellation methods, can skip Section~\ref{sec:small-cancellation} and accept Proposition~\ref{Lem:Zquot} as a ``black box". A complete proof of Theorem~\ref{thm:main} modulo this proposition is given in Section 4.

\subsection*{Acknowledgements} The authors would like to thank the anonymous referee for carefully reading a draft of this paper, and for a number of suggestions that improved the exposition.

\section{Preliminaries}
\subsection{Weak relative hyperbolicity}\label{subsec:weak_rel_hyp} Recall that metric space $T$ with a distance function $\d$ is said to be \emph{geodesic}, if every two points $a,b\in T$ can be connected by a path of length $\d(a,b)$. Given a path $p$ in a metric space, we denote by $p_-$ and $p_+$ its initial and  terminal points, respectively. If $p$ is rectifiable, we denote by $\ell(p)$ its length.

In this paper, we consider graphs (possibly with loops and multiple edges) as metric spaces with respect to the standard distance function obtained by identifying the interior of each edge with the open unit interval $(0,1)$. We say that a path $p$ in a graph $\Gamma$ is \emph{combinatorial} if $p$ is a concatenation of edges of $\Gamma$.

A geodesic metric space $T$ is  \emph{hyperbolic} if there is  $\delta\ge 0$  such that for any geodesic triangle $\Delta $ in $T$, every side of $\Delta $ is contained in the union of the closed $\delta$-neighborhoods of the other two sides \cite[III.H.1]{BH}. A finitely generated group $G$ is {\it hyperbolic} if the Cayley graph $\Gamma (G,X)$ is hyperbolic for some (equivalently, any) finite generating set $X$ of $G$.

To define the relative version of hyperbolicity used in this paper, we need a few auxiliary notions. Let $G$ be a group. A \emph{generating alphabet} $\mathcal A$ of a group $G$ is an abstract set given together with a (not necessarily injective) map $\alpha\colon \mathcal A\to G$ whose image generates $G$. For example, every generating set $X$ of $G$ can be thought of as a generating alphabet with the natural inclusion map $X\to G$. We say that a word $W=a_1\ldots a_n$, where $a_1,\ldots, a_n\in \mathcal A$, \emph{represents} an element $g$ in the group $G$ if the equality $g=\alpha(a_1)\ldots \alpha(a_n)$ holds in $G$. To simplify our notation, we omit $\alpha$ and do not distinguish between letters of the alphabet $\mathcal A$ and elements of $G$ represented by them whenever no confusion is possible.

By the \emph{Cayley graph} of $G$ with respect to a generating alphabet $\mathcal A$, denoted $\Gamma (G, \mathcal A)$, we mean a graph with the vertex set $G$ and the set of edges defined as follows. For every $a\in \mathcal A$ and every $g\in G$, there is an oriented edge $e$ going from $g$ to $ga$ in $\Gamma (G, \mathcal A)$ and labelled by $a$. Given a combinatorial path $p$ in $\Gamma (G, \mathcal A)$, we denote by $\lab (p)$ its label and by $p^{-1}$ the combinatorial inverse of $p$. We use the notation $\d_{\mathcal A}$ and $|\cdot|_{\mathcal A}$ to denote the standard metric on $\Gamma (G,\mathcal A)$ and the word length on $G$ with respect to  (the image of) $\mathcal A$, respectively.

In this paper, generating alphabets will occur in the following settings. Let $\Hl$ be a collection of subgroups of a group $G$. A subset $X\subseteq G$ is called a \emph{relative generating set} of $G$ with respect to $\Hl$ if $X$ and the union of all subgroups $H_i$ together generate $G$. We think of $X$ and the subgroups $H_i$ as abstract sets and consider the disjoint unions
\begin{equation}\label{calA}
\mathcal H = \bigsqcup\limits_{i\in I} H_i\quad {\rm and}\quad \mathcal A= X \sqcup \mathcal H.
\end{equation}

\begin{conv}
Henceforth, we assume that all generating sets and relative generating sets are symmetric, i.e., closed under inversion.
\end{conv}

This convention implies that the alphabet $\mathcal A$ defined by (\ref{calA}) is also symmetric since so are $X$ and each $H_i$.  In particular, every element of $G$ can be represented by a word in the alphabet $\mathcal A$.

We can naturally think of the Cayley graphs $\Gamma (H_i, H_i)$ as subgraphs of $\CGXH$. For every $i\in I$, we introduce a generalized metric \[\dl \colon H_i \times H_i \to [0, \infty]\] as follows.
\begin{figure}
  \begin{center}
      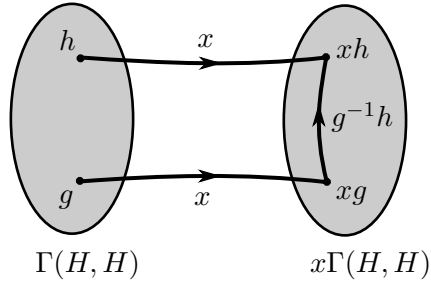
        \end{center}
   \caption{A path of length $3$ in Example \ref{bex}.}\label{fig0}
 \end{figure}
 
\begin{defn}
Given $g,h\in H_i$, let $\dl (g,h)$ be the length of a shortest path in the Cayley graph $\CGXH$ that connects $g$ to $h$ and contains no edges of $\Gamma (H_i, H_i)$. If no such path exists, we set $\dl (h,k)=\infty $. We call $\dl$ the {\it generalized metric on $H_i$ associated to the triple $(G, \Hl, X)$.}
\end{defn}

Clearly, $\dl $ satisfies the triangle inequality, where addition is extended to $[0, \infty]$ in a natural way: $a+\infty =\infty $, for all $a\in [0, \infty]$. 

\begin{defn}[Dahmani--Guirardel--Osin, {\cite{DGO}}]\label{hedefn}
A group $G$ is {\it weakly hyperbolic} relative to a collection of subgroups $\Hl$ and a relative generating set $X$ if $\CGXH$ is hyperbolic.  If, in addition, the set \[\{ h\in H_i\mid \dl(1,h)\le n\}\] is finite for every $n\in \NN$ and every $i\in I$, we say that the collection of subgroups $\Hl$ is \emph{hyperbolically embedded} in $G$ with respect to $X$ and write $\Hl\h (G,X)$. We also write $\Hl\h G$ and say that $\Hl$ is {\it hyperbolically embedded} in $G$ if $\Hl\h (G,X)$ for some relative generating set $X$.
\end{defn}

To help the reader comprehend these definitions, we provide some standard examples.

\begin{ex}
For any group $G$ we have $G\h G$.  Indeed we can take $X=\emptyset $ in this case. Then $\CGXH=\Gamma (G,G)$ and $\widehat\d_G (g,h)=\infty $ for any $g\ne h$. Further, if $H$ is a finite subgroup of $G$, then we have $H\h (G,X)$ for $X=G$.
\end{ex}

\begin{ex}\label{bex}
Let $G=H\times \mathbb Z$, where $H$ is an infinite group, and let  $x$ denote a generator of $\mathbb Z$. It is easy to see that the graph $\Gamma (G, \{ x\}\sqcup H)$ is quasi-isometric to a line. In particular, $\Gamma (G, \{ x\}\sqcup H)$ is hyperbolic and thus $G$ is weakly hyperbolic relative to $H$. However, every two elements $g,h\in H$ can be connected by a path of length at most $3$ in $\Gamma (G, \{ x\}\sqcup H)$ that avoids edges of $\Gamma (H,H)$ (see Fig. \ref{fig0}). By definition, we have $\widehat\d_H(g,h)\le 3$ for all $g,h\in H$ and, therefore, $H$ is not hyperbolically embedded in $G$ with respect to $X$. 
\end{ex}

The lemma below easily follows from an elaborated version of the argument given in Example~\ref{bex}.

\begin{lem}[Dahmani--Guirardel--Osin, {\cite[Proposition 4.33]{DGO}}]\label{Lem:maln}
Let $G$ be a group, $\Hl $ a hyperbolically embedded collection of subgroups of $G$.
\begin{enumerate}
\item[\rm (a)] For any $i\in I$ and any $g\in G\setminus H_i$, we have $|H_i \cap g^{-1}H_i g|<\infty $.
\item[\rm (b)] For any distinct $i,j\in I$ and any $g\in G$, we have $|H_i \cap g^{-1}H_jg|<\infty$.
\end{enumerate}
\end{lem}

We also record an elementary yet useful observation.

\begin{lem}\label{Lem:d*}
Let $G$ be a group, $\{H_i\}_{i\in I}$ a collection of subgroups of $G$, $X$ a relative generating set of $G$ with respect to $\Hl$. Suppose that $\e\colon G\to G^\ast$ is a surjective homomorphism such that
\begin{equation}\label{Eq:ker}
\ker \e\le \bigcap_{i\in I} H_i.
\end{equation}
Let $H_i^\ast=\e(H_i)$, $X^\ast=\e(X)$, and let $\dl$ (respectively, $\dl^\ast$) denote the generalized metric on $H_i$ associated to the triple $(G, \Hl, X)$ (respectively, $(G^\ast, \{ H_i^\ast\}_{i\in I}, X^\ast)$). In this notation, the following hold.
\begin{enumerate}
\item[\rm (a)] $G$ is weakly hyperbolic relative to $\Hl$ and $X$ if and only if $G^\ast$ is weakly hyperbolic relative to $\{H_i^\ast\}_{i\in I}$ and $X^\ast$.
\item[\rm (b)] For every $i\in I$ and every $h\in H_i$, we have
$\dl(h,1)\ge \dl^\ast(\e(h), 1).$
\end{enumerate}
\end{lem}

\begin{proof}
Let ${\mathcal H}=\bigsqcup_{i\in I} H_i$ and ${\mathcal H^\ast}=\bigsqcup_{i\in I} H_i^\ast$. The map $\e\colon G\to G^\ast$ naturally extends to a surjective graph morphism \[\gamma\colon \CGXH\to \Gamma (G^\ast, X^\ast \sqcup \mathcal H^\ast),\] sending edges labelled by a letter $x\in X$ (respectively, $h\in H_i$) to edges labelled by the letter $\e(x)\in X^\ast$ (respectively, $\e(h)\in H_i^\ast$). Let $\d$ and $\d^\ast$ denote the word metrics on $G$ and $G^\ast$ with respect to the generating alphabets $X \sqcup \mathcal H$ and $X^\ast \sqcup \mathcal H^\ast$, respectively. The inclusion (\ref{Eq:ker}) guarantees that $\d (1,k)\le 1$ for all $k\in \ker \e$. This implies that 
\[
\d^\ast (\e(a),\e(b))\le \d(a,b) \le \d^\ast(\e(a),\e(b))+1, \text{ for all } a,b \in G.
\]
In particular, $\gamma$ is a quasi-isometry and (a) follows by \cite[Theorem 1.9, III.H.1]{BH}.

To prove (b), let $h\in H_i$ for some $i\in I$. If $\dl (h,1)=\infty$, the desired inequality vacuously holds. Hence, we can assume that there is a path $p$ of length $\ell(p)=\dl(h,1)$ in $\CGXH$ connecting $1$ to $h$ and avoiding edges of the subgraph $\Gamma (H_i, H_i)$. The map $\gamma$ takes $p$ to a path $p^\ast $ in the graph $\Gamma (G^\ast, X^\ast \sqcup \mathcal H^\ast)$ of the same length connecting $1$ to $\e(h)$. Since $\ker \e\leqslant H_i$, the path $p^\ast $ avoids edges of $\Gamma (H_i^\ast, H_i^\ast)$ for every $i\in I$. The desired inequality now follows from the definitions of $\dl$ and $\dl^\ast$.
\end{proof}

\subsection{$H_i$-components and geodesic polygons in relative Cayley graphs} As above, let $G$ be a group, $\Hl$ a collection of subgroups of $G$, $X$ a relative generating set of $G$ with respect to $\Hl$. The following terminology was introduced in \cite{DGO} (in the particular case of relatively hyperbolic groups it goes back to \cite{Osi06}).

\begin{defn}\label{Def:comp}
Let $p$ be a combinatorial path in the Cayley graph $\CGXH $ and let $i\in I$. A non-trivial subpath $q$ of $p$ is called an {\it $H_i $-component} (or simply a \emph{component}) of $p$ if the label of $q$ is a word in the alphabet $H_i\subseteq X \sqcup \mathcal{H}$ and $q$ is not contained in any longer subpath of $p$ with this property. Two $H_i$-components $q_1, q_2$ of the same path $p$ (or of two distinct paths $p_1$ and $p_2$) are {\it connected} if all vertices of $q_1$ and $q_2$ belong to the same coset of $H_i$ in $G$. An $H_i$-component $q$ of a path $p$ is called {\it isolated } if no distinct $H_i$-component of $p$ is connected to $q$.
\end{defn}

By a \emph{geodesic $n$-gon} in a metric space we mean a closed circuit consisting of $n$ geodesic segments. The lemma below is a simplification of \cite[Proposition 4.14]{DGO}.

\begin{lem}\label{Omega}
Let $G$ be a group, $\Hl$ a collection of subgroups of $G$, $X$ a relative generating set of $G$ with respect to $\Hl$. Suppose that  $\CGXH$ is hyperbolic. Then there exists a constant $D$ satisfying the following condition. For any combinatorial geodesic $n$-gon $p$ in $\CGXH $, any $i\in I$ and any isolated $H_i$-component $q$ of $p$, the element $h\in H_i$ represented by the label of $q$  satisfies $\dl (1, h)\le Dn$.
\end{lem}

\subsection{Acylindrically hyperbolic groups and suitable subgroups}
According to Bow\-ditch \cite{Bow}, an isometric action of a group $G$ on a metric space $(T,\d)$ is said to be {\it acylindrical} if, for every $\e>0$ there exist $R,N>0$ such that for every pair of points $x,y\in S$, with $\d (x,y)\ge R$, there are at most $N$ elements $g\in G$ satisfying the inequalities
\[
\d(x,gx)\le \e \quad {\rm and}\quad \d(y,gy) \le \e.
\]
Note that every isometric group action on a  bounded metric space is trivially acylindrical. 
The following definition was proposed in \cite{Osi16}.

\begin{defn}
A group $G$ is \emph{acylindrically hyperbolic} if admits a non-elementary acylindrical action on a hyperbolic space.
\end{defn}

Recall that an action of a group $G$ on a hyperbolic space $T$ is \emph{non-elementary} if the limit set of $G$ on the Gromov boundary $\partial T$ has infinitely many points (equivalently, that $G$ contains two loxodromic elements with distinct sets of fixed points on $\partial T$). By \cite[Theorem 1.1]{Osi16}, an acylindrical action of a group $G$ on a hyperbolic space is non-elementary if and only if $G$ is not virtually cyclic and has unbounded orbits. Readers unfamiliar with the notions of Gromov's boundary and the limit set can accept this as the definition of a non-elementary action.

The class of acylindrically hyperbolic groups includes all non-(virtually cyclic) hyperbolic groups, mapping class groups of closed surfaces of non-zero genus, ${\rm Aut} (F_n)$ and ${\rm Out}(F_n)$ for $n\ge 2$ (where $F_n$ is the free group of rank $n$), non-(virtually cyclic) groups acting properly and cocompactly on irreducible Hadamard manifolds that are not higher rank symmetric spaces, groups of deficiency at least $2$, most $3$-manifold groups, and many other examples. For more details, we refer to the survey \cite{Osi18} and references therein.

Acylindrical hyperbolicity of a group is closely related to the existence of hyperbolically embedded subgroups. This relationship will play an important role in our paper. The next theorem follows from the work of the third author in \cite{Osi16}.

\begin{thm}\label{Thm:ah-joined} Let $G$ be a group.
\begin{itemize}
    \item[(i)] If $G$ is acylindrically hyperbolic then $G$ contains an infinite proper hyperbolically embedded subgroup.
    \item[(ii)] Suppose that  $\Hl$ is a collection of infinite proper subgroups of $G$, $Y$ is a relative generating set of $G$ with respect to $\Hl$ and $\Hl\h (G, Y)$. Then there exists another relative generating set $X$ of $G$ with respect to $\Hl$ such that $Y\subseteq X$, $\Hl\h (G,X)$ and the action of $G$ on $\Gamma (G, X\cup \mathcal H)$ is acylindrical and non-elementary, where $\mathcal H$ is the alphabet defined by \eqref{calA}. In particular, $G$ is acylindrically hyperbolic.
\end{itemize}
 \end{thm}

\begin{proof} Claim (i) was proved in \cite[Theorem 1.2]{Osi16}. In claim (ii) the existence of a relative generating set $X$ of $G$ such that $Y\subseteq X$, $\Hl\h (G,X)$ and the action of $G$ on $\Gamma (G, X\cup \mathcal H)$ is acylindrical is given by  \cite[Theorem 5.4]{Osi16}. Since the family $\Hl$ consists of infinite proper subgroups of $G$, $G$ contains a loxodromic element with respect to its action on the Cayley graph $\Gamma (G, X\cup \mathcal H)$ by \cite[Corollary~5.3]{AMS}. Moreover, by \cite[Theorem~6.14]{DGO}, $G$ contains non-abelian free subgroups, hence it cannot be virtually cyclic. We can now apply \cite[Theorem~1.1]{Osi16} to conclude that the action of $G$ on $\Gamma (G, X\cup \mathcal H)$ must be non-elementary, as required.  
\end{proof}

For any group $G$, we denote by $\mathcal {AH}(G)$ the set of all generating sets $X$ of $G$ such that the Cayley graph $\Gamma (G, X)$ is hyperbolic and the natural action of $G$ on $\Gamma (G, X)$ is acylindrical. Note that $G\in \mathcal {AH}(G)$ for every $G$. Given $X\in \mathcal {AH}(G)$, a subgroup $S\leqslant G$ is said to be \emph{non-elementary} with respect to $X$ if the induced action of $S$ on $\Gamma (G, X)$ is non-elementary. Clearly, the existence of a non-elementary subgroup of $G$ with respect to some $X\in \mathcal {AH}(G)$ implies that $G$ is acylindrically hyperbolic.

Further, a subgroup $S\leqslant G$ is  \emph{suitable} with respect to $X\in \mathcal {AH}(G)$ if $S$ is non-elementary with respect to $X$ and does not normalize any non-trivial finite subgroup of $G$. If $G$ is torsion-free, the properties of being non-elementary and suitable with respect to a given $X\in \mathcal {AH}(G)$ are obviously equivalent. 

The lemma below is a simplified version of \cite[Corollary~5.7]{Hull}.

\begin{lem}[Hull, \cite{Hull}]\label{Lem:hee}
Let $G$ be a group, let $X\in\mathcal {AH}(G)$, and let $S$ be a suitable subgroup of $G$ with respect to $X$. For every $n\in \NN$, there exist infinite cyclic subgroups $H_1, \ldots, H_n$ of $S$ such that $\{ H_1, \ldots, H_n\}\h (G, X)$.
\end{lem}

By Theorem~\ref{Thm:ah-joined}~(i) and \cite[Theorem~2.24]{DGO} every acylindrically hyperbolic group contains a unique maximal finite normal subgroup, denoted by $K(G)$. The next lemma provides us with an important source of suitable subgroups.

\begin{lem}[{\cite[Lemma 3.23]{CIOS}}]\label{lem:normal->suitable}
Let $G$ be a group and let $X \in \mathcal{AH}(G)$ be such that the $G$-action on  the Cayley graph $\Gamma(G,X)$ is   non-elementary. If  $K(G)=\{1\}$, then every non-trivial normal subgroup of $G$ is suitable with respect to $X$.    
\end{lem}

\section{Small cancellation theory over weakly relatively hyperbolic groups}
\label{sec:small-cancellation}

The principal goal of this section is to develop a method for constructing words with small cancellations in weakly relatively hyperbolic groups and apply it to central extensions of acylindrically hyperbolic groups. Our main results are Propositions \ref{Prop:SCW} and \ref{Lem:Zquot}. Although we tried to make our exposition as self-contained as possible, the proofs heavily rely on the techniques developed in \cite{DGO,Ols93,Osi10,Hull} and other papers. 

\subsection{Small cancellation conditions}
We begin by recalling the definition of the small cancellation condition $C_1(\e, \mu, \rho)$ considered in \cite{CIOS}. This is a simplified version of the condition $C_1(\e, \mu, \lambda, c, \rho)$ introduced in \cite{Ols93} corresponding to $\lambda =1$ and $c=0$.

Let $\mathcal A$ be a generating alphabet of a group $G$. We write $U\equiv V$ to express the equality of two words in the alphabet $\mathcal A$; the equality $U=V$ will mean that $U$ and $V$ represent the same element of the group $G$. A word $W$ in the alphabet $\mathcal A$ is said to be {\it geodesic} if any path in the Cayley graph $\Gamma (G, \mathcal A)$ labeled by $W$ is geodesic.  If a word $W$ decomposes as $W\equiv UV$ for some words $U$, $V$  in the alphabet $\mathcal A$, we say that $U$ is an \emph{initial subword} of $W$.  The length of a word $W$ (i.e., the number of letters in $W$) is denoted by $\| W\|$.

Let $\mathcal R$ be a set of words in the alphabet $\mathcal A$. We say that $\mathcal R$ is {\it symmetric} if, for any $R\in \mathcal R$, $\mathcal R$ contains all cyclic shifts of $R^{\pm 1}$. If $\mathcal R$ is not symmetric, we define its \emph{symmetrization} to be the smallest symmetric set of words in the alphabet $\mathcal A$ containing $\mathcal R$.

\begin{defn}\label{DefSC}
Let $\e \ge 0$, $\mu\in (0,1)$ and  $\rho \in \NN$. A symmetric set of words $\mathcal R$  in the alphabet $\mathcal A$ satisfies the \emph{small cancellation condition $C_1(\e, \mu, \rho )$ over $G$}, if the following conditions hold.
\begin{enumerate}
\item[(a)] All words in $\mathcal R$ are geodesic and have length at least  $\rho $.
\item[(b)] Suppose that two distinct words $R,R^\prime \in \mathcal R$ have initial subwords $U$ and $U^\prime$, respectively,  such that
\begin{equation}\label{piece1}
\max \{ \| U\| ,\, \| U^\prime\| \} \ge \mu \min\{\| R\|, \| R^\prime\|\}
\end{equation}
and $U^\prime = YUZ$ in $G$, for some words $Y$, $Z$ in the alphabet $\mathcal A$ of length
\begin{equation}\label{piece2}
\max \{ \| Y\| , \,\| Z\| \} \le \e.
\end{equation}
Then $YRY^{-1}=R^\prime$ in $G$.
\end{enumerate}
\begin{enumerate}
\item[(c)] Suppose that a word $R\in \mathcal R$ contains two disjoint subwords $U$ and $U^\prime$ such that either $U^\prime = YUZ$ or $U^\prime = YU^{-1}Z$ in $G$, for some words $Y$, $Z$ in the alphabet $\mathcal A$ satisfying inequality (\ref{piece2}). Then
    \[ \max \{ \| U\| ,\, \| U^\prime\| \}< \mu \| R\|.\]
\end{enumerate}
\end{defn}

\begin{rem}\label{Rem:simult}
Note that condition $C_1(\e, \mu, \rho)$ becomes stronger as $\e$ and $\rho$ increase and $\mu$ decreases. In other words, given any parameters $\e_i \ge 0$, $\mu_i \in (0,1)$ and $\rho_i \in \NN$, $i=1,\dots,k$, the small cancellation condition $C_1(\e,\mu,\rho)$ implies the conditions 
$C_1(\e_i,\mu_i,\rho_i)$, for each $i=1,\dots,k$, where $\e= \max\{\e_i \mid 1 \le i \le k\}$, $\mu=\min\{\mu_i \mid 1 \le i \le k\}$ and $\rho=\max\{e_i \mid 1 \le i \le k\}$.
\end{rem}

\subsection{Words with small cancellation}
In practice, verifying the small cancellation condition $C_1(\e, \mu, \rho )$ is a non-trivial task. Our next goal is to show that words satisfying certain  combinatorial conditions also enjoy $C_1(\e, \mu, \rho )$. 

For the remainder of this subsection, we fix a group $G$, a collection of subgroups $\Hl$, where $|I|\ge 2$, and a relative generating set $X$ of $G$ with respect to $G$ such that $G$ is weakly hyperbolic relative to $\Hl$ and  $X$. We keep the notation introduced in (\ref{calA}) and denote by $\dA$ the distance function in the Cayley graph $\CGXH$. Recall also that all generating sets are supposed to be symmetric by default; in particular, we have $X=X^{-1}$.

Let $p=p_1up_2vp_3$ be a path in $\CGXH$ such that $u$ and $v$ are components of $p$ (see Definition \ref{Def:comp}). We say that $u$ and $v$ are {\it consecutive components} of $p$ if $p_2$ does not contain any components; if $p_2$ is the trivial path, $u$ and $v$ are called {\it strongly consecutive}. This definition extends to any number of components in the obvious way. We will need the following.

\begin{lem}[Dahmani--Guirardel--Osin, {\cite[Lemma 4.21 (b)]{DGO}}]\label{421}
Let $\mathcal W$ denote the set of all words $W$ in the alphabet $\mathcal A=X\sqcup \mathcal H$ satisfying the following conditions for all $i,j\in I$.
\begin{enumerate}
    \item[\rm (W$_1$)] No two consecutive letters of $W$ belong to $X$ or the same $H_i$.
    \item[\rm (W$_2$)] If some letter $a\in H_i$ occurs in $W$ then $\dl(1,a)> 50D$, where $D$ is the constant from Lemma~\ref{Omega}.
    \item[\rm (W$_3$)] If $axb$ is a subword of $W$, where $x\in X$, $a\in H_i$ and $b\in H_j$, then either $i\ne j$ or the element represented by $x$ in $G$ does not belong to $H_i$.
\end{enumerate}
For every $\e>0$ and $K\in \NN$, there exists $L=L(\e, K)$ such that the following holds. Suppose that $p$ and $q$ are paths in $\CGXH$ such that $\ell(p)\ge L$, the words $\lab(p)$, $\lab (q)$ belong to $\mathcal W$, and
    \[
    \max\left\{\dA(p_-, q_-),\, \dA(p_+, q_+)\right\} \le \e.
    \]
Then there exist consecutive components $u_1, \dots, u_K$ of the path $p$ and consecutive components $v_1, \dots, v_K$ of the path $q$ such that $u_s$ is connected to $v_s$ for $s=1, \ldots, K$.
\end{lem}

\begin{rem}
Recall that we consider the alphabet $\mathcal{A}=X \sqcup \mathcal{H}$ as an abstract set such that the natural map $\mathcal{A} \to G$ (restricting to the identity on $X$ and on each $H_i$) need not be injective. In particular, some  $x \in X$ may represent the same element of $G$ as some $h \in H_i$. This shows that condition  (W$_3$) from Lemma~\ref{421} is not a consequence of condition  (W$_1$).
\end{rem}

It is also proved in \cite[Lemma 4.21~(a)]{DGO} that paths in $\Gamma(G, X\sqcup \mathcal H)$ labeled by words from the set $\mathcal W$ are uniformly quasi-geodesic. For our purposes, we will need a stronger result obtained in \cite{CIOS}.

\begin{lem}[Chifan--Ioana--Osin--Sun, {\cite[Lemma 3.10]{CIOS}}]\label{Wgeod}
Suppose that $p$ is a path in $\CGXH$ such that
\[
\lab(p)\equiv U_1 a_{1} U_2a_{2}\ldots U_{n}a_nU_{n+1},
\]
and the following conditions hold.
\begin{enumerate}
\item[\rm (a)] For every $j=1, \ldots, n$,  $a_{j}$ is a letter in $H_{i(j)}$ for some $i(j)\in I$ and we have \\ $\widehat\d_{H_{i(j)}}(1, a_{j}) > 5D$, where $D$ is the constant provided by Lemma \ref{Omega}.
\item[\rm (b)] For every $j=1,\ldots , n+1$, $U_j$ is a (possibly empty) word in the alphabet $\mathcal A$ such that, for any element $g\in G$ satisfying $H_{i(j-1)}gH_{i(j)}=H_{i(j-1)}U_jH_{i(j)}$, we have $\| U_j\| \le |g|_{\mathcal A}$. Here we assume $H_{i(0)}=H_{i(n+1)}=\{ 1\}$ for convenience.
\item[\rm (c)] If $U_j$ is the empty word for some $j=2,\ldots , n$, then $H_{i(j-1)}\ne H_{i(j)}$.
\end{enumerate}
Then $p$ is geodesic.
\end{lem}

From now on, we fix two distinct subgroups $H_1,H_2\in \Hl$. Variants of the following proposition have previously appeared in \cite{Osi10,Hull,CIOS} in the particular case when $\{H_1, H_2\}$ are hyperbolically embedded in $(G,X)$. A slightly modified version of the same proof as in these papers also works in our settings.

\begin{prop}\label{Prop:SCW}
Let $\e\ge 0$, $\mu \in (0,1)$, $\rho \in \NN$ be any constants and let $\{ x_i\mid i\in K\}$ be a subset of $X\setminus H_2H_1$. Further, let $\mathcal R$ denote the symmetrization of a set of words of the form
\[
x_ia_{i1}b_{i1}\ldots a_{im}b_{im}, \quad i\in K,
\]
where $m$ is some natural number and the following conditions hold.
\begin{enumerate}
\item[\rm (a)] For all $i\in K$ and $j\in \{ 1, \ldots, m\}$, we have $a_{ij}\in H_1$, $b_{ij}\in H_2$, and
\begin{equation}\label{Eq:dlab}
\min\{ \widehat \d_1(1, a_{ij}), \,\widehat \d_2(b_{ij}, 1)\}\ge 50D,
\end{equation}
where $D$ is the constant from Lemma \ref{Omega}.
\item[\rm (b)]  If $a_{k\ell}a_{ij}^{\gamma }\in H_1\cap H_2$ or $b_{k\ell}b_{ij}^{\gamma }\in H_1\cap H_2$ for some $i,k \in K$, $j, \ell\in \{ 1, \ldots, m\}$, and $\gamma=\pm 1$, then $i=k$, $j=\ell$ and $\gamma=-1$.
\end{enumerate}
Suppose also that
\begin{equation}\label{*}
    H_1\cap H_2\leqslant C_G(H_1)\cap C_G(H_2)\cap \left (\bigcap_{i\in K} C_G(x_i)\right).
\end{equation}
Then $\mathcal R$ satisfies the condition $C_1(\e,\mu, \rho)$  over $G$ for all sufficiently large $m$.
\end{prop}

\begin{figure}
   \begin{center}
   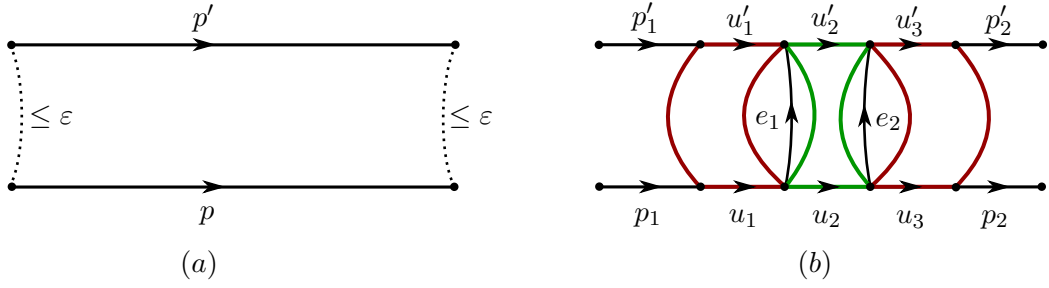
   \end{center}
   \vspace*{-3mm}
   \caption{Paths $p$, $p^\prime$, and their consecutive connected components}\label{fig1}
 \end{figure}
 
\begin{proof}
Let
\begin{equation}\label{Eq:m}
m\ge \max\left\{\rho,  \lceil L/\mu\rceil\right\},
\end{equation}
where $L=L(\e, 7)$ is the constant from Lemma \ref{421}.

We first note that every path $p$ in $\CGXH$ labeled by a word $R\in \mathcal R$ has the structure described in Lemma \ref{Wgeod} (condition (b) of the lemma follows immediately from the assumption that $x_i\notin H_2H_1$). Hence, $R$ is geodesic. Note also that we have $\| R\| = 2m+1 >\rho$. Thus, part (a) of Definition \ref{DefSC} is satisfied.

Further, suppose that there are two relations $R,R^\prime \in \mathcal R$ such that $R\equiv UV$, $R^\prime \equiv U^\prime V^\prime$,
$U^\prime = YUZ$ in $G$ for some words $Y$, $Z$, and inequalities (\ref{piece1}), (\ref{piece2}) hold. Without loss of generality, we can assume that $\| U^\prime \| \le \| U\|$. Combining this with (\ref{Eq:m}) and (\ref{piece1}), we obtain
\begin{equation}\label{U'RR'}
\| U\| \ge \mu \min\{ \| R\|, \| R^\prime\| \}=\mu (2m+1)> L.
\end{equation}
Translating our assumptions to a geometric language, we can find paths $p$ and $p^\prime$ in the Cayley graph $\CGXH$ such that
\[
\lab(p)\equiv U,\quad \lab (p^\prime )\equiv U^\prime,
\]
and
\[
\max\left\{\dA (p_-, p^\prime _-), \dA (p_+, p^\prime_+)\right\} \le \e
\]
(see Fig. \ref{fig1}~(a)).

Observe that $\mathcal R\subseteq \mathcal W$, where the set $\mathcal W$ is defined in Lemma \ref{421}. Indeed, conditions (W$_1$) and (W$_3$) hold by inspection and (W$_2$) is ensured by (\ref{Eq:dlab}). By (\ref{U'RR'}), we also have $\ell(p)= \|U\| >L$. Therefore, by Lemma \ref{421} and the choice of $L$, there exist $7$ consecutive components of $p$ that are connected to $7$ consecutive components of $p^\prime$. Since $p$ and $p^\prime$ contain at most one edge that is not a component, we can pick $3$ strongly consecutive components of $p$ that are connected to $3$ strongly consecutive components of $p^\prime$. That is,
\[
p=p_1 u_1u_2u_3p_2 \quad \text{and} \quad p^\prime=p_1^\prime u_1^\prime u_2^\prime u_3^\prime p_2^\prime ,
\]
where $u_s$ is connected to $u_s^\prime $ for $s=1,2,3$ (see Fig. \ref{fig1}~(b)). Without loss of generality, we can assume that $u_1$, $u_3$, $u_1^\prime $, $u_3^\prime $ are $H_1$-components while $u_2$ and $u_2^\prime $ are $H_2$-components.

For $s=1,2$, let $e_s$ be a path in $\CGXH$ connecting $(u_s)_+$ to $(u_s^\prime )_+$ and let $g_s$ be the element of $G$ represented by $\lab(e_s)$. By the definition of connected components, we have
\begin{equation}\label{Eq:g1g2}
    g_1, g_2\in H_1\cap H_2.
\end{equation}
Without loss of generality, we can assume that $\lab(u_2)=b_{ij}$ and $\lab(u_2^\prime)=b_{k\ell}^\beta$ for some $\beta=\pm 1$. Reading the label of the loop $e_1u_2^\prime e_2^{-1}u_2^{-1}$, we obtain the equality $g_1b_{k\ell}g_2^{-1}b_{ij}^{-\beta}=1$. Using (\ref{*}) and (\ref{Eq:g1g2}), we obtain $b_{k\ell}b_{ij}^{-\beta}\in H_1\cap H_2$. Hence, $i=k$, $j=\ell$, and $\beta=1$ by assumption (b). This means that the cyclic shifts of $R$ and $R^\prime$ starting from $b_{ij}=b_{k\ell}$ coincide.  That is, we have
\[
U\equiv R_1b_{ij}R_2, \quad U^\prime \equiv R_1^\prime b_{k\ell}R_2^\prime ,
\]
for some words $R_1, R_2, R_1^\prime, R_2^\prime$ and
\begin{equation}\label{Eq:bij}
b_{ij}R_2VR_1\equiv  b_{k\ell}R_2^\prime V^\prime R_1^\prime.
\end{equation}

\begin{figure}
   \begin{center}
   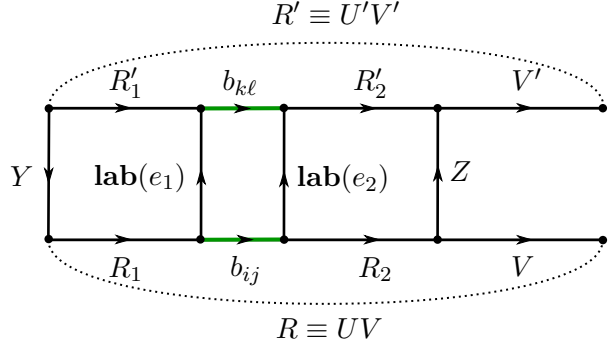
   \end{center}
   \vspace*{-3mm}
   \caption{Labels of paths involved in the proof of the equality $YRY^{-1}=R'$.}\label{fig2}
 \end{figure}

The equality $U^\prime =YUZ$ in $G$ gives rise to a configuration in $\CGXH$ depicted on Fig. \ref{fig2}. Reading the label of the leftmost quadrilateral, we obtain that the word 
\begin{equation}\label{Eq:C}
C\equiv R_1^{-1}Y^{-1}R_1^\prime
\end{equation}
represents the element $g_1$ in $G$. By (\ref{*}) and (\ref{Eq:g1g2}), the element $g_1$ commutes with all letters of $R$ and $R^\prime$ in $G$. Therefore, using (\ref{Eq:C}) and (\ref{Eq:bij}), we obtain the following equalities in $G$:
\begin{equation*}
\begin{split}
YRY^{-1} & =  YR_1b_{ij}R_2VY^{-1}= R_1^\prime C^{-1} b_{ij}R_2VR_1C(R_1^\prime)^{-1} =  R_1^\prime b_{ij}R_2VR_1(R_1^\prime)^{-1} \\
& =  R_1^\prime b_{k\ell} R_2^\prime V^\prime R_1^\prime(R_1^\prime)^{-1} =R_1^\prime b_{k\ell} R_2^\prime V^\prime = R^\prime.
\end{split}
\end{equation*}
Thus, $\mathcal R$ satisfies part (b) of Definition \ref{DefSC}.

Finally, suppose that some word $R\in \mathcal R$ contains two disjoint subwords $U$ and $U^\prime$ such that $U^\prime = YUZ$ or $U^\prime = YU^{-1}Z$ in $G$ for some words $Y$, $Z$ and the inequalities (\ref{piece1}) (where $R^\prime \equiv R$), (\ref{piece2}) hold. Arguing as above, we can find two disjoint occurrences of a letter from $H_1\sqcup H_2$ or its inverse in $R$. However,  this contradicts condition (b).  Thus, part (c) of Definition \ref{DefSC} also holds.
\end{proof}

\subsection{An application to central extensions of acylindrically hyperbolic groups}
For a subset $\mathcal R$ of a group $G$, we denote by $\ll \mathcal R \rr$ the normal closure of $\mathcal R$ (i.e., the smallest normal subgroup of $G$ containing $\mathcal R$). If the Cayley graph $\Gamma (G, \mathcal A)$ is hyperbolic, the small cancellation conditions  $C_1(\e , \mu , \rho )$ can be used to control properties of the quotient $G/\ll \mathcal R\rr$. This idea is used in the proof of Proposition~\ref{Lem:Zquot} below, which relies on a number of results obtained in \cite{CIOS,Hull,Osi10}.  Some of these results were proved under the assumption that $\mathcal R$ satisfies the small cancellation conditions $C(\e, \mu, \rho)$, $C(\e, \mu, \lambda, c, \rho)$ or $C_1(\e, \mu, \lambda, c, \rho)$. The definitions of these conditions are not important for us and we refer the interested reader to \cite{Ols93,CIOS} for details. For our purpose, it suffices to know that all these conditions are weaker than the condition $C_1(\e, \mu, \rho)$ (with the same values of the parameters)  discussed above. We summarize the necessary results in two lemmas below.

\begin{lem}\label{Lem:SCQ1}
Let $G$ be a group with a generating alphabet $\mathcal A$ such that the Cayley graph $\Gamma (G, \mathcal A)$ is hyperbolic. For every $N\in \NN$, there exist $\e\ge 0$, $\mu \in (0,1)$ and $\rho \in \NN$ such that for any symmetric set of words $\mathcal R$ in the alphabet $\mathcal A$ satisfying $C_1(\e, \mu, \rho)$ the following hold.
\begin{enumerate}
\item[\rm (a)] {\rm (Osin, \cite{Osi10})} \,The restriction of the natural homomorphism $\eta \colon G\to Q=G/\ll \mathcal R\rr$ to the subset $\{ g\in G\mid |g|_{\mathcal A}\le N\}$ is injective.
\item[\rm (b)] {\rm (Chifan--Ioana--Osin--Sun, \cite{CIOS})}\,  For every element $g \in G$ such that $|g|_{\mathcal A}\le N$, we have $C_{Q}(\eta(g))=\eta(C_G(g))$.
\end{enumerate}
\end{lem}
\begin{proof}[About the proof] 
By Remark~\ref{Rem:simult}, it suffices to find  (possibly different) $\e\ge 0$, $\mu \in (0,1)$, and $\rho \in \NN$ so that each of the parts (a), (b) holds individually. 

The proof of part (a) repeats the proof of part 2) of Lemma 5.1 in \cite{Osi10} verbatim. Indeed, although Lemma 5.1 in \cite{Osi10} is formally stated under a stronger assumption that $G$ is hyperbolic relative to some collection of subgroups, this condition is not used in the proof of the second part of the lemma; its proof only involves \cite[Lemma 4.4]{Osi10}, which is stated and proved under the assumption that $\Gamma (G, \mathcal A)$ is hyperbolic. Part (b) is proved in \cite[Lemma~3.5]{CIOS}.
\end{proof}

The next lemma summarizes \cite[Lemma 4.4 and Lemma 4.9]{Hull} (Remark~\ref{Rem:simult} also applies here). 

\begin{lem}[Hull, \cite{Hull}]\label{Lem:SCQ2}
Let $G$ be a group, $\Hl$ a collection of subgroups of $G$, $X$ a relative generating set of $G$ with respect to $\Hl$. Let $\mathcal H$ be the alphabet defined by (\ref{calA}). Suppose that $\Hl\h (G, X)$ and the action of $G$ on $\Gamma (G, X\cup \mathcal H)$ is acylindrical.  Then there exist $\e\ge 0$, $\mu \in (0,1)$, and $\rho \in \NN$ such that, for any finite symmetric set of words $\mathcal R$ in the alphabet $X\cup\mathcal H$  satisfying $C_1(\e, \mu, \rho)$, the following hold.
\begin{enumerate}
\item[\rm (a)] The restriction of the natural homomorphism  $\eta \colon G\to Q=G/\ll \mathcal R\rr$ to each $H_i$ is injective and we have $\{\eta(H_i)\}_{i\in I}\h (Q,\eta(X))$.
\item[\rm (b)] For every $n\in \NN$, every element of $Q$ of order $n$ is the image of an element of $G$ of order $n$ under $\eta$.  In particular, if $G$ is torsion-free then so is $Q$.
\end{enumerate}
\end{lem}

We note that the original statement of Lemma~\ref{Lem:SCQ2} (a), as written in \cite[Lemma~4.4]{Hull}, does not mention the generating set $\eta(X)$ explicitly and simply states that $\{\eta(H_i)\}_{i\in I}\h Q$. However, it is clear from the argument that the family of subgroups $\{\eta(H_i)\}_{i\in I}$ is hyperbolically embedded in $Q$ with respect to the image of the generating set $X$.

\begin{notat}\label{notat:ast}
Throughout the rest of this section, we employ the following notation. For a group $G$, we denote by $Z(G)$ its center and let $G^\ast=G/Z(G)$. The image of a subset or a subgroup $S\subseteq G$ (respectively, an element $g\in G$) in $G^\ast$ under the natural homomorphism $G\to G^\ast$ is denoted by $S^\ast$ (respectively, $g^\ast$).    

For any epimorphism $\eta\colon G\to Q$ between groups $G$ and $Q$, clearly $\eta(Z(G))\leqslant Z(Q)$. Therefore, $\eta $ induces an epimorphism $\eta^\ast\colon G^\ast\to Q^\ast=Q/Z(Q)$ such that the following diagram  is commutative (the vertical arrows are the natural homomorphisms):
\begin{equation}\label{comm_diag}
 \begin{tikzcd}
G \arrow[r, "\eta"] \arrow[d]
& Q \arrow[d] \\
G^\ast \arrow[r, "\eta^\ast"]
& Q^\ast
\end{tikzcd}
\end{equation}
\end{notat}

Recall that for an acylindrically hyperbolic group $G$ we use $K(G)$ to denote the maximal finite normal subgroup of $G$. In this notation, we have the following.

\begin{prop}\label{Lem:Zquot}
Let $G$ be a group, let $Y$ be a (possibly infinite) generating set of $G$, and let $S$ be a subgroup of $G$ containing $Z(G)$. Suppose that $Y^\ast \in \mathcal{AH} (G^\ast)$ and $S^\ast$ is a suitable subgroup of $G^\ast$ with respect to $Y^\ast$. Then, for any non-zero $N\in \NN$ and any finite subset $T\subseteq G$, there exist a group $Q$ and a surjective homomorphism $\eta\colon  G\to Q$ satisfying the following conditions.
\begin{enumerate}
\item[\rm (a)] $\eta (T)\subseteq \eta(S)$ and $\ker \eta$ is contained in the normal closure of $T\cup S$ in $G$.  
\item[\rm (b)] The group $Q^\ast$ is acylindrically hyperbolic and $K(Q^\ast)=\{ 1\}$. In fact, there is a generating set $V$ of $Q$ such that $\eta(Y) \subseteq V$, $V^\ast \in \mathcal{AH}(Q^\ast)$ and the action of $Q^\ast$ on $\Gamma(Q^\ast,V^\ast)$ is non-elementary.
\item[\rm (c)] For every element $g \in G$ such that $|g| _{Y\cup Z(G)}\le N$, we have $C_{Q}(\eta(g))=\eta(C_G(g))$.
\item[\rm (d)] $\eta (Z(G))=Z(Q)$.
\item[\rm (e)] For all elements $g\in \ker \eta \setminus \{ 1\}$, we have $|g| _{Y\cup Z(G)}> N$. In particular, the restriction of $\eta$ to $Z(G)$ is injective.
\item[\rm (f)] For every $n\in \NN$, every element of $Q^\ast$ of order $n$ is the image of an element of $G^\ast$ of order $n$ under the homomorphism $\eta^\ast\colon G^\ast\to Q^\ast$. In particular, if $G^\ast $ is torsion-free then so is $Q^\ast$.
\item[\rm (g)] For every $n\in \NN$, every element of $Q$ of order $n$ is the image of an element of $G$ of order $n$ under the homomorphism $\eta$. In particular, if $G$ is torsion-free then so is $Q$.
\end{enumerate}
In particular, if we additionally assume $G$ to be finitely generated, then there exist a group $Q$ and a homomorphism $\eta\colon  G\to Q$ such that $\eta(S)=Q$ and condition (b)--(g) hold.
\end{prop}

\begin{proof}
By Lemma \ref{Lem:hee}, we can find two infinite cyclic subgroups $H_1^\ast, H_2^\ast\leqslant S^\ast$ such that $\{ H_1^\ast, H_2^\ast\}\h (G^\ast, Y^\ast)$. By \cite[Corollary~4.27]{DGO}, we can add any finite set of elements to $Y^\ast$ without violating the condition $\{ H_1^\ast, H_2^\ast\}\h (G^\ast, Y^\ast)$. In particular, we can assume that 
\begin{equation}\label{Eq:TinY}
T^\ast\cup\{1\}\subseteq Y^\ast
\end{equation}
without loss of generality. 

By Theorem~\ref{Thm:ah-joined}~(ii), there exists a relative generating set $X^\ast$ of $G^\ast$ with respect to $\{ H_1^\ast, H_2^\ast\}$ such that $Y^\ast  \subseteq X^\ast$, $\{ H_1^\ast, H_2^\ast\}\h (G^\ast, X^\ast)$ and the action of $G^\ast $ on the Cayley graph $\Gamma (G^\ast, X^\ast\sqcup H_1^\ast\sqcup H_2^\ast)$ is acylindrical. For $i=1,2$, we denote by $\dl^\ast$ the generalized distance on $H_i^\ast $ associated to the triple $(G^\ast, \{ H_1^\ast, H_2^\ast\}, X^\ast)$.

Let $X$ be the full preimage of $X^\ast$ and let $H_i\leqslant S$ be the full preimage of $H_i^\ast$ in $G$ under the natural homomorphism $G\to G^\ast$, $i=1,2$. By (\ref{Eq:TinY}), we have $T\subseteq X$ and $Z(G) \subseteq X$. It  is also clear that $X$ is a relative generating set of $G$ with respect to $\{ H_1, H_2\}$. Denote by $\dl$ the generalized metric on $H_i$ associated to the triple $(G, \{ H_1, H_2\}, X)$. Note that $G$ is weakly hyperbolic relative to $\{ H_1, H_2\}$ and $X$ by Lemma \ref{Lem:d*}.

Let 
\begin{equation}\label{Eq:propZ0}
T\setminus H_2 H_1= \{ x_1, \ldots, x_r\}.
\end{equation}
Note that if $x_i^\ast =h_2^\ast h_1^\ast$, for some $i\in \{1, \ldots, r\}$, then $h_1\in H_1$ and $h_2\in H_2$, so that $x_i=h_2h_1z$, where $z\in Z(G)\leqslant H_1$. Hence, $x_i\in H_2H_1$, which contradicts (\ref{Eq:propZ0}). Thus, we have 
\begin{equation}\label{Eq:propZ1}
\{ x_1^\ast, \ldots, x_r^\ast \}\cap H_2^\ast H_1^\ast =\emptyset .
\end{equation}

Fix any $N\in \NN$. By Remark \ref{Rem:simult}, we can find $\e\ge 0$, $\mu \in (0,1)$ and $\rho \in \NN$ such that:
\begin{itemize}
\item[(i)] the claim of Lemma \ref{Lem:SCQ1} holds for $G$ and its generating alphabet $\mathcal A=X\sqcup H_1\sqcup H_2$;
\item[(ii)] the claim of Lemma \ref{Lem:SCQ2} holds for $G^\ast$, the collection of subgroups $\{ H_1^\ast, H_2^\ast\}$, and the relative generating set $X^\ast$.
\end{itemize}

Since $H_1^\ast\cong H_2^\ast\cong \mathbf Z$ and $\{ H_1^\ast, H_2^\ast\}\h (G^\ast, X^\ast)$, we have 
\begin{equation}\label{Eq:propZ2}
H_1^\ast\cap H_2^\ast =\{ 1\},
\end{equation}
by Lemma \ref{Lem:maln}. Further, since $H_1^\ast$, $H_2^\ast$ are infinite and hyperbolically embedded in $G^\ast$, we can choose elements $a_{ij}^\ast\in H_1^\ast$, $b_{ij}^\ast\in H_2^\ast$, $1\le i\le n$, $j\in \NN$, such that
\begin{equation}\label{Eq:propZ3}
\min\{ \widehat \d_1^\ast(a_{ij}^\ast,1), \,\widehat \d_2^\ast(b_{ij}, 1)\}> 50D,
\end{equation}
for all $i$ and $j$, where $D$ is the constant given by Lemma~\ref{Omega}, and
\begin{equation}\label{Eq:propZ4}
a_{ij}^\ast \ne (a_{k\ell}^\ast)^{\pm 1}, \quad b_{ij}^\ast \ne (b_{k\ell}^\ast)^{\pm 1}
\end{equation}
whenever $(i,j)\ne (k, \ell)$. Since $\{ H_1^\ast, H_2^\ast\}\h (G^\ast,X^\ast)$, the group $G^\ast$ is weakly hyperbolic relative to $\{ H_1^\ast, H_2^\ast\}$ and $X^\ast$. This and (\ref{Eq:propZ1})--(\ref{Eq:propZ4}) allow us to apply Proposition \ref{Prop:SCW} and conclude that the symmetrization of 
\[
\mathcal R^\ast=\{ x_i^\ast a_{i1}^\ast b_{i1}^\ast \ldots a_{im}^\ast b_{im}^\ast \mid i=1, \ldots, r\}
\]
satisfies the small cancellation condition $C_1(\e, \mu, \rho)$ over $G^\ast$, for all sufficiently large $m\in \NN$.

Furthermore, let $a_{ik}\in H_1$, $b_{ik}\in H_2$ be arbitrary preimages of $a_{ik}^\ast$ and $b_{ik}^\ast$ under the natural homomorphism $G\to G^\ast$. Using Lemma \ref{Lem:d*} and the inclusion $H_1\cap H_2\leqslant Z(G)$ implied by (\ref{Eq:propZ2}), it is straightforward to check that all the assumptions of Proposition \ref{Prop:SCW} also hold for the set of words
\[
\mathcal R =\{ x_i a_{i1} b_{i1} \ldots a_{im} b_{im} \mid i=1, \ldots, r\}.
\]
Thus, the symmetrization of $\mathcal R$  satisfies the  small cancellation condition $C_1(\e, \mu, \rho)$ over $G$ for all sufficiently large $m\in \NN$. From now on, we fix a sufficiently large $m\in \NN$ so that the symmetrizations of both $\mathcal R$ and $\mathcal R^\ast$ satisfy $C_1(\e, \mu, \rho)$ over the corresponding groups.

Let
$Q=G/\ll \mathcal R\rr
$
and let $\eta\colon G\to Q$ be the natural homomorphism. Obviously, the group $Q$ is a central extension of the form
\begin{equation}\label{eq:shs_for_Q}
 \{1\}\longrightarrow \eta (Z(G)) \longrightarrow Q \longrightarrow G^\ast/\ll \mathcal R^\ast\rr \longrightarrow \{1\}.   
\end{equation}

Below, we will show that $\eta(Z(G))=Z(Q)$, so that $G^\ast/\ll \mathcal R^\ast\rr= Q/Z(Q)=Q^\ast$. Until this equality is established, we will use the notation $Q^\circledast=G^\ast/\ll \mathcal R^\ast\rr$; $\eta^\ast:G^\ast \to Q^\circledast$ will denote the natural homomorphism.

We first note that property (a) obviously holds. Indeed, the relations in $\mathcal R$ ensure that
\[
\eta(x_i)\in \langle \eta(H_1\cup H_2)\rangle \subseteq \langle \eta(S)\rangle, 
\]
for all $i=1, \ldots, r$; by (\ref{Eq:propZ0}), this implies the inclusion $\eta(T)\subseteq \eta(S)$. Since $\mathcal R\subseteq \langle T\cup H_1\cup H_2\rangle$, we have $\ker \eta \subseteq \ll T\cup S\rr$.

Further, by (i) and (ii), we can apply Lemmas \ref{Lem:SCQ1} and \ref{Lem:SCQ2} to $G/\ll \mathcal R\rr $ and $G^\ast/\ll \mathcal R^\ast\rr $, respectively. 
By Lemma~\ref{Lem:SCQ2} (a), we have 
\[\eta^\ast(H_1^\ast)\cong \eta^\ast(H_2^\ast)\cong \mathbf Z \quad\text{and}\quad \{\eta^\ast(H_1^\ast), \eta^\ast(H_2^\ast)\}\h (Q^\circledast, \eta^\ast(X^\ast)).\] Using Theorem~\ref{Thm:ah-joined}~(ii), we can find a generating set $V^* \in \mathcal{AH}(Q^\circledast)$ such that the action of $Q^\circledast$ on $\Gamma(Q^\circledast,V^\ast)$ is non-elementary. In particular,  $Q^\circledast$ is acylindrically hyperbolic.

The finite subgroup $K(Q^\circledast)$ must be contained in every hyperbolically embedded subgroup of $Q^\circledast$ by \cite[Theorem~2.24~(b)]{DGO}. Hence $K(Q^\circledast)=\{ 1\}$. The center of any acylindrically hyperbolic group is finite by \cite[Corollary~7.2~(a)]{Osi16}. Therefore, $Z(Q^\circledast)\leqslant K(Q^\circledast)$ and consequently $Z(Q^\circledast)=\{1\}$. In view of (\ref{eq:shs_for_Q}), the latter implies that $Z(Q)\leqslant \eta (Z(G))$; the opposite inclusion is obvious and we obtain (d). Thus, we have $Q^\circledast=Q^\ast$. To complete the proof of (b), let $V \subseteq Q$ to be the full preimage of $V^\ast$. Since $Y \subseteq X$ and $\eta^\ast(X^\ast) \subseteq V^\ast$, by construction, the commutative diagram \eqref{comm_diag} implies that $\eta(Y) \subseteq V$.

Now, recall that $Y\cup Z(G)\subseteq X \subseteq \mathcal A$ by construction, and, therefore, every element $g\in G$ satisfies $|g|_{Y\cup Z(G)}\ge |g|_{\mathcal A}$. This and Lemmas~\ref{Lem:SCQ1}, \ref{Lem:SCQ2} yield parts (c), (e), (f).

Finally, to prove (g), assume that $q\in Q$ is an element of finite order $n$. Then $(q^\ast)^n=1$. By (f), there exists $g\in G$ such that $(g^\ast)^n=1$ and $q^\ast=\eta^\ast(g^\ast)=\eta(g)^\ast$ (see (\ref{comm_diag})). The latter equality implies that $q=\eta(g)t$, for some $t\in Z(Q)$. By (d), there is $z\in Z(G)$ such that $t=\eta (z)$. Obviously, we have $\eta(gz)=\eta(g)\eta(z)=q$. Since \[((gz)^n)^\ast=((gz)^\ast)^n=(g^\ast)^n=1,\] we obtain $(gz)^n\in Z(G)$. The equality $\eta(1)=1=q^n=\eta((gz)^n)$ and injectivity of $\eta$ on $Z(G)$ imply $(gz)^n=1$. Taking into account that the order of $gz$ must be divisible by the order of $\eta(gz)=q$, we conclude that it equals $n$, as required.
\end{proof}

\begin{rem} \label{rem:T} We stated Proposition~\ref{Lem:Zquot} above for an arbitrary finite subset $T \subset G$ with a view to future applications. However in this paper, we will only use it in the situation when $G$ is finitely generated and $T$ is a finite generating set of $G$. In this case claim (a) amounts to saying that $\eta(S)=\eta(G)=Q$, i.e., $\eta$ is surjective on $S$.    
\end{rem}

\section{Proof of the main theorem}

\subsection{An auxiliary countable simple group}
 As explained in the introduction, the $p$-adic Lie groups from Theorem~\ref{thm:main} arise as closed subgroups of the automorphism group of a certain discrete simple group $G_\infty$. We begin by constructing this group.
 
\begin{lem}\label{lem:Halls_gp}
For every countable abelian group $Z$ there exists a $3$-step solvable $2$-generated group $B$ such that 
\begin{itemize}
    \item[(a)] $Z(B) \cong Z$;
    \item[(b)] $Z(B) \subseteq [B,B]$;
    \item[(c)] $B/Z(B) \cong \Z \wr \Z$;
    \item[(d)] if $Z$ is torsion-free then so is $B$.
\end{itemize}
\end{lem}

\begin{proof} We will use a classical construction of P. Hall \cite{Hal}. By \cite[Theorem~7]{Hal} there is a $2$-generated $3$-step solvable group $A$ such that $Z(A) \cong \Z^\omega$ is the free abelian group of infinite countable rank. Moreover, by construction, $A/Z(A) \cong \Z \wr \Z$ is torsion-free and has trivial center, and $Z(A) \subseteq [A,A]$ (see \cite[Subsection~3.2]{Hal}).

Since $Z$ is a countable abelian group and $Z(A) \cong \Z^\omega$, there is a normal subgroup $N \lhd Z(A)$ such that $Z \cong Z(A)/N$. Since $N$ is central in $A$, we can define the group $B=A/N$. Clearly $B$ will be $2$-generated and solvable of derived length at most $3$. Since $A/Z(A)$ is centerless, we have $Z(B)=Z(A)/N \cong Z$, which also implies that $Z(B) \subseteq [A,A]/N \subseteq [B,B]$. 

Finally, observe that $B/Z(B) \cong A/Z(A) \cong \Z \wr \Z$ is torsion-free. Therefore, if $Z(B) \cong Z$ is torsion-free then the same will be true for $B$.    
\end{proof}

\begin{lem}\label{lem:centralizers_in_amalgams}
Let $G=A*_C B$ be the free amalgamated product of groups $A$ and $B$ over a common subgroup $C$ such that $C$ is proper and central in each of them. Then $C_G(a)=C_A(a)$, for every $a \in A \setminus C$, and $Z(G)=C$.
\end{lem}

\begin{proof}
Let $\mathcal T$ be the Bass-Serre tree for $G$ corresponding to the amalgamated product decomposition given in the statement. Then there is a vertex $v$ of $\mathcal T$ whose stabilizer is $A$, and the stabilizers of the edges incident to $v$ are conjugates of $C$ in $A$ (all of which are equal to $C$ because $C$ is central in $A$). Thus if $a \in A \setminus C$ is any element, then $a$ fixes $v$ and does not fix any edge incident to it, hence $v$ is the only vertex of $\mathcal T$ fixed by $a$. It follows that any $g \in C_G(a)$ also fixes $v$, which implies that $g \in A$. Therefore $C_G(a) \subseteq A$, so  $C_G(a)=C_A(a)$, for every $a \in A \setminus C$. Similarly, $C_G(b)=C_B(b)$, for every $b \in B \setminus C$.

To prove the last claim, choose some elements $a \in A \setminus C$ and $b \in B \setminus C$ (which is possible because $C \neq A$ and $C \neq B$ by the assumptions). Then
\[Z(G) \subseteq C_G(a) \cap C_G(b)=C_A(a) \cap C_B(b) \subseteq A \cap B=C.\] The inclusion $C \subseteq Z(G)$ is obvious, thus $Z(G)=C$.
\end{proof}

\begin{lem}\label{lem:emb_P_in_G}
Let $P$ be a finitely generated group and let $Z \leqslant Z(P)$ be a central subgroup of $P$. Then $P$ can be embedded as a subgroup in a finitely generated group $G$ with the following properties.
\begin{itemize}
    \item[\rm (a)]  $G$ is perfect, i.e., $[G,G]=G$.
    \item[\rm (b)]  $Z(G)=Z$.
    \item[\rm (c)] for any subset $Q \subseteq P$, with $Q \not\subseteq Z$, one has $C_G(Q)=C_P(Q)$. 
    \item[\rm (d)] $G/Z(G)$ is a simple group.
    \item[\rm (e)] If $P$ is torsion-free then so is $G$.
    \item[\rm (f)] For each $h \in P \setminus Z$ we have $C_G(h)=C_P(h)$.
\end{itemize}
\end{lem}

\begin{proof} Without loss of generality we can assume that $Z$ is a proper subgroup of $P$. Otherwise, conditions (c) and (f) become vacuous and we can prove the remaining statements of the lemma by considering $P \times \Z$ instead of $P$.

Note that $Z$ is a countable abelian group as it is central in the finitely generated group $P$. Let $B$ be the $2$-generated group given by Lemma~\ref{lem:Halls_gp}, so that 
$Z \cong Z(B)\subseteq [B,B]$, $B/Z(B)$ is infinite and if $Z$ is torsion-free then so is $B$.

Fix an isomorphism $\alpha:Z \to Z(B)$ and define the group $G_0$ by the following presentation: 
\[G_0=\langle P,B \mid z=\alpha(z), \text{ for all } z \in Z \rangle.\] In other words, $G_0$ is  the amalgamated free product $P*_{Z=\alpha(Z)} B$.

Clearly the group $G_0$ is finitely generated, and it is torsion-free if and only if $P$ is torsion-free.
Lemma~\ref{lem:centralizers_in_amalgams} tells us that
\begin{equation}\label{eq:props_of_G_0}
 Z(G_0)=Z \text{ and } C_{G_0}(h)=C_P(h), \text{ for every } h \in P \setminus Z.     
\end{equation}

Observe that 
\begin{equation}\label{eq:quot_is_free_prod}
G_0/Z \cong (P/Z)*(B/Z(B)).    
\end{equation}
 Since $P/Z$ is non-trivial and $B/Z(B)$ is infinite, it follows that 

\begin{itemize}
    \item $G_0/Z$ is not virtually cyclic and is hyperbolic relative to $\{P/Z,B/Z(B)\}$;
     \item the action of $G_0/Z$ on the Cayley graph $\Gamma\Bigl(G_0/Z,\e_0(Y_0)\Bigr)$ is acylindrical and non-elementary, where $Y_0=P \cup B$ and $\e_0:G_0 \to G_0/Z(G_0)$ is the natural homomorphism (see \cite[Proposition~5.2]{Osi16});
       \item $K(G_0/Z)=\{1\}$;
\end{itemize}
 Let $g_1,g_2,\dots$ be an enumeration of all elements of $G_0$. Set  $N_0=[G_0,G_0] \lhd G_0$ and  $N_i=\ll g_i \rr Z(G_0) \lhd G_0$, for $i \in \N$.

We are now going to construct a sequence of groups $G_0,G_1,,G_2, \dots$, together with epimorphisms $\eta_i:G_{i-1} \to G_{i}$, such that the following conditions  hold for all $i \in \N$ (we use $\xi_i: G_0 \to G_i$ to denote the composition $\eta_i \circ \eta_{i-1} \circ \dots \circ \eta_1$):
\begin{enumerate}[label={\normalfont (C\arabic*)}]
     \item \label{cond:1} $Z(G_i)=\eta_i(Z(G_{i-1}))$;
    \item \label{cond:2}  there is a generating set $Y_i$ of $G_i$ such that $\xi_i(Y_0) \subseteq Y_i$, $\e_i(Y_i) \in \mathcal{AH}(G_i/Z(G_i))$, where  $\e_i:G_i \to G_i/Z(G_i)$ is the natural homomorphism, and $G_i/Z(G_i)$ is non-elementary with respect to $\e_i(Y_i)$;
    \item \label{cond:3}  $K(G_i/Z(G_i))=\{1\}$; 
    \item \label{cond:4} if $P$ is torsion-free then so is $G_i$;
    \item \label{cond:5}  the restriction of $\eta_i$ to  $P_{i-1}=\xi_{i-1}(P)$ is injective;
    \item \label{cond:6}  $C_{G_i}(\eta_i(h))=\eta_i(C_{G_{i-1}}(h))$, for every $h \in P_{i-1}$.
\end{enumerate}

To simplify our notation, we denote $G_{-1}=G_0$ and let $\eta_0=\xi_0:G_0 \to G_0$ be the identity map. Suppose that the groups $G_0,\dots,G_j$ and the epimorphisms $\eta_0,\dots,\eta_j$, satisfying \ref{cond:1}--\ref{cond:6},  have already been constructed, for some $j \ge 0$. Observe that $Z(G_j)=\xi_j(Z(G_0))$ by \ref{cond:1} and $Z(G_0) \subseteq N_j$ by construction; for $j=0$ this follows from the fact that \[Z(G_0)=Z(B) \subseteq [B,B] \subseteq N_0.\]

If $\xi_j(N_{j}) = Z(G_j)$ then we set $G_{j+1}=G_j$ and let $\eta_{j+1}:G_j \to G_{j+1}$ be the identity map. Obviously in this case $G_{j+1}$ will enjoy conditions \ref{cond:1}--\ref{cond:6} for $i=j+1$.

Now, suppose that $S=\xi_j(N_{j})$ properly contains $Z(G_j)$ (in particular, this is true when $j=0$, because $\e_0(S)=\e_0([G_0,G_0])=[G_0/Z,G_0/Z]$ is non-trivial as the non-trivial free product \eqref{eq:quot_is_free_prod} is not abelian).
Then $\e_j(S)$ is a non-trivial normal subgroup of $G_j/Z(G_j)$,  hence $\e_j(S)$ is a suitable subgroup of $G_j/Z(G_j)$ with respect to $\e_j(Y_j)$ by \ref{cond:2}, \ref{cond:3} and Lemma~\ref{lem:normal->suitable}. Therefore, we can apply Proposition~\ref{Lem:Zquot} and Remark~\ref{rem:T}, to find a group $G_{j+1}$ and an epimorphism $\eta_{j+1}:G_j \to G_{j+1}$ satisfying $\eta_{j+1}(\xi_j(N_j))=G_{j+1}$ and conditions \ref{cond:1}--\ref{cond:6} for $i=j+1$ (conditions \ref{cond:5} and \ref{cond:6} are achieved by taking any $N \in \N$ in Proposition~\ref{Lem:Zquot} and observing that $|\e_j(h)|_{\e_j(Y_j)} \le 1 \le N$ for every $h \in P_j$, which holds because $P_j=\xi_j(P)\subseteq \xi_j(Y_0) \subseteq Y_j$ by \ref{cond:2}). 

Thus we have inductively defined the desired sequence 
$G_0 \stackrel{\eta_1}{\to} G_1 \stackrel{\eta_2}{\to} \dots $. We can now let $G$ be the direct limit of this sequence. This means that $G \cong G_0/M$, where $\displaystyle M=\bigcup_{i \in \N } \ker\xi_i \lhd G_0$ and $\xi_i=\eta_i \circ \dots \circ \eta_1:G_0 \to G_i$, for $i \in \N$. Let $\psi: G_0 \to G$ denote the natural homomorphism, with $\ker \psi=M$. Since $\ker\xi_i \subseteq \ker\xi_{i+1}$, for all $i \in \N$, there is an epimorphism $\psi_i:G_i \to G$ such that 
\begin{equation} \label{eq:psi_i}
\psi=\psi_i \circ \xi_i, \text{ for }i \in \N.    
\end{equation}

Clearly $G$ is finitely generated, as a quotient of $G_0$. Let us check that $G$ satisfies the desired properties (a)--(f). The map $\xi$ is injective on $P$ by condition \ref{cond:5}, so $P$ can be viewed as a subgroup of $G$. The group $G$ is perfect because $G_1$ is perfect, as $G_1=\eta_1(N_0)$ and $N_0 = [G_0,G_0]$ by construction.

Condition \ref{cond:4} implies that $G$ is torsion-free whenever $P$ is, and condition \ref{cond:6} together with \eqref{eq:props_of_G_0} yield that
\begin{equation}\label{eq:centralizers_of_P-elts}
    C_G(\psi(h))=\psi(C_{G_0}(h))=\psi(C_P(h)),  \text{ for every } h \in P \setminus Z,
\end{equation} so claim (f) of the lemma holds.

It remains to establish claims (b),(c) and (d). For the former, it is clear that \[\psi(Z)=\psi(Z(G_0)) \subseteq Z(G).\] If $z \in Z(G)$ then, as $G$ is finitely generated and a direct limit of the sequence $(G_i,\eta_i)$, there will be $i \in \N$ and $z_i \in Z(G_i)$ such that $z=\psi_i(z_i)$. Condition \ref{cond:1} implies that $Z(G_i)=\xi_i(Z(G_0))=\xi_i(Z)$, hence $z \in \psi(Z)$ by \eqref{eq:psi_i}. Thus $Z(G)=\psi(Z)$ and (b) holds.

For (c), assume that $Q$ is a subset of $P$ not contained in $Z$. For any $h \in Q \setminus Z$, in view of \eqref{eq:centralizers_of_P-elts},  we have 
\[C_G(\psi(Q)) \subseteq C_G(\psi(h))=\psi(C_P(h)) \subseteq \psi(P).\]
Therefore $C_G(\psi(Q))=C_{\psi(P)}(\psi(Q))$, and we can use the injectivity of $\psi$ on $P$ to obtain
\[C_G(\psi(Q))=C_{\psi(P)}(\psi(Q))=\psi(C_P(Q)),\]
 confirming claim (c).

Finally, let us verify claim (d). Suppose that $N^\ast \lhd G/Z(G)$ is a non-trivial normal subgroup; let $N \lhd G$ denote its full preimage in $G$. Then $Z(G)$ is strictly contained in $N$, so there must exist $j \in \N$ such that $\psi(g_j) \in N \setminus Z(G)$. It follows that $\psi(N_j) \subseteq N$ and  $Z(G) \subsetneqq \psi(N_j)$. In view of \eqref{eq:psi_i}, we can conclude that $Z(G_j)  \subsetneqq \xi_j(N_j)$. Therefore $\xi_{j+1}(N_j)=\eta_{j+1}(\xi_j(N_j))=G_{j+1}$, by construction, whence $\psi(N_j)=G$ by \eqref{eq:psi_i}. Therefore $N=G$, which implies that $N^\ast=G/Z(G)$. The latter shows that $G/Z(G)$ is a simple group, completing the proof of the lemma.
\end{proof}

\begin{prop}\label{prop:asc-chain}
For any countable group $R$ and arbitrary integers $d, p\geq 1$ there exists an infinite ascending chain of groups $G_1\leqslant G_2\leqslant \ldots $ such that the following conditions hold for all $n\in \NN$.
\begin{enumerate}
    \item[\rm (i)] $G_n$ is finitely generated and perfect. 
    \item[\rm (ii)] $G_n/Z(G_n)$ is simple, and there is an isomorphism $\psi_n \colon Z(G_n)\to \ZZ^d$.
    \item[\rm (iii)] $Z(G_{n+1}) \leqslant Z(G_n)$ and $\psi_n\big(Z(G_{n+1})\big) = p\ZZ^d$.
    \item[\rm (iv)] $C_{G_{n}}(G_1)=Z(G_1)$.
    \item[\rm (v)] $R$ is a subgroup of $G_1$, with $R \cap Z(G_1)=\{1\}$.
    \item[\rm (vi)] If $R$ is torsion free then so is $G_n$.
\end{enumerate}
If in addition $p \geq 2$, then the countable group $\displaystyle G_\infty = \bigcup_{n=1}^\infty G_n$ is simple. 
\end{prop}

\begin{proof} By a celebrated theorem of Higman, Neumann and Neumann \cite{HNN}, every countable group can be embedded in a group generated by two elements, so, without loss of generality, we can assume that $R$ is finitely generated. The required sequence of groups will be constructed inductively, using Lemma~\ref{lem:emb_P_in_G}. Indeed, by applying this lemma to the case when $P =R \times Z$, where $Z \cong \ZZ^d$ is the free abelian group of rank $d$ freely generated by $\{a_1, \dots, a_d\}$, we obtain a finitely generated perfect group $G_1$ such that $Z(G_1)=Z$ and $G_1/Z(G_1)$ is simple. Moreover, $R \leqslant G_1$, $R \cap Z(G_1)=\{1\}$ and $G_1$ is torsion-free provided $R$ is torsion-free.

Now we can apply Lemma~\ref{lem:emb_P_in_G} to the case when $P=G_1$ and $Z=\langle a_1^p, \dots, a_d^p \rangle$ to construct a group $G_2$ satisfying properties (i), (ii) and (vi) for $n=2$. We also see that $Z(G_2)=\langle a_1^p, \dots, a_d^p\rangle$ and $Z(G_1)=\langle a_1, \dots, a_d \rangle$, so that (iii) holds. Property~(iv) is guaranteed by condition (c) of Lemma~\ref{lem:emb_P_in_G}, where we take $Q=G_1$.

To construct $G_3$, we take $P=G_2$ and $Z=\langle a_1^{p^2}, \dots, a_d^{p^2}\rangle$ in Lemma~\ref{lem:emb_P_in_G}. Then $C_{G_3}(G_1)=C_{G_2}(G_1)=Z(G_1)$ by condition (c) and induction. Proceeding in the same manner, we  obtain an infinite sequence of groups $G_1 \leqslant G_2 \leqslant G_3 \leqslant\dots$, which will enjoy the properties (i)--(vi).    

Finally, suppose that $p \geq 2$ and consider a non-trivial normal subgroup $N \lhd G_\infty$, where $G_\infty = \bigcup_n G_n$. Then $N \cap G_n \neq \{1\}$ for all sufficiently large $n$. Let $\e_n \colon G_n \to G_n/Z(G_n)$ be the canonical projection. From (ii), it follows that if $\e_n(N \cap G_n) \neq G_n/Z(G_n)$, then $N \cap G_n \subseteq Z(G_n)$. For $m > n$ sufficiently large, (iii) implies that $\e_m(N \cap G_n)$ is non-trivial, as, by hypothesis,  $p \geq 2$. Therefore we have $\e_m(N \cap G_m) = G_m/Z(G_m)$ or, equivalently,  $G_m=(N \cap G_m) Z(G_m)$. It follows that $G_m/(N\cap G_m) \cong Z(G_m)/(N \cap Z(G_m)) $ is abelian. Since $G_m$ is perfect by (i), we infer that the quotient $G_m/(N\cap G_m)$ is trivial, hence $G_m = N \cap G_m$. This shows that $N$ contains $G_m$ for all sufficiently large $m$. Thus $N = G_\infty$, so $G_\infty$ is simple.
\end{proof}

\subsection{From $G_\infty$ to $p$-adic Lie groups}
Recall that, given any locally compact group $G$, the  group $\Aut(G)$ of homeomorphic automorphisms of $G$ carries a natural group topology, called the \emph{Braconnier topology}, that can be defined as the coarsest group topology on $\Aut(G)$ such that the action map $\Aut(G) \times G \to G$ is continuous (see \cite[Appendix~I]{CM11} and references therein). In the case where $G$ is discrete, which is the only case to be considered in this paper, the Braconnier topology on $\Aut(G)$ coincides with the topology of pointwise convergence for the action on $G$ (see \cite[Satz 1.6]{GLM70}). 

The closure of $\Inn(G)$ in $\Aut(G)$, denoted by $\Linn(G)$, was first considered by Gol'berg \cite{Gol46}, and is sometimes called the group of \emph{locally inner automorphisms}. It consists of those automorphisms of $G$ that act as an inner automorphism on each finitely generated subgroup. By construction, it is a topological group that is totally disconnected.
The action of $G$ on itself by conjugation induces a natural (continuous) homomorphism $G \to \Linn(G)$.
The following proposition summarizes basic properties of $\Linn(G)$.

\begin{lem}\label{lem:Linn}
For each discrete group $G$, the topological group $\Linn(G) = \overline{\Inn(G)}$ has the following properties. 
\begin{enumerate}
\item[\rm (i)] Given $\Gamma \leqslant G$, the image of $\Gamma$ in $\Linn(G)$ has compact closure if and only if the conjugation action of $\Gamma$ on $G$ has finite orbits. 

    \item[\rm (ii)] $\Linn(G)$ is locally compact if and only if there is a finite subset $F \subseteq G$ such that the conjugation action of $C_G(F)$   on $G$ has finite orbits. 

    \item[\rm (iii)] $\Linn(G)$ is discrete if and only if there is a finite subset $F \subseteq G$ such that $C_G(F) = Z(G)$.

    \item[\rm (iv)] If $G$ is simple and non-abelian, then $\Linn(G)$ is topologically simple.
\end{enumerate}
\end{lem}

\begin{proof}
    For (i), the forward implication directly follows from the continuity of the action of $\Linn(G)$ on the discrete group $G$. For the reverse implication, see \cite[Proposition~I.7]{CM11} or \cite[Theorem~0.1]{GLM70}. By definition of the topology, the group $\Linn(G)$ has a basis of identity neighborhoods consisting of the pointwise stabilizers of finite subsets of $G$, and  each of those stabilizers is an open subgroup. Hence $\Linn(G)$ is locally compact if and only if there is a finite subset $F \subseteq G$ such that the pointwise stabilizer $\mathrm{Fix}_{\Linn(G)}(F)$ is compact. Since $\mathrm{Fix}_{\Linn(G)}(F)$ is open and since $\Inn(G)$ is dense in $\Linn(G)$, it follows that the image of $C_G(F)$ is dense in $\mathrm{Fix}_{\Linn(G)}(F)$. Hence $\mathrm{Fix}_{\Linn(G)}(F)$ is compact if and only if the closure of the image of $C_G(F)$ is compact. Therefore, assertion (ii) follows from (i). 
    
    For (iii) and (iv), we refer respectively to  Proposition I.8 (i)  and Lemma I.5 (ii) in \cite{CM11}. 
\end{proof}

\begin{thm}\label{thm:technical-version}
Suppose that $d \geq 1$ and $p\geq 2$ are any integers. Let $R$ be any countable group and let $G_1 \leqslant G_2 \leqslant \dots$ be the ascending chain of groups provided by Proposition~\ref{prop:asc-chain}. Let also $G_\infty = \bigcup_{n\in \N} G_n$. Then the topological group  $\mathcal G = \Linn(G_\infty) = \overline{\Inn(G_\infty)}$ has the following properties. 
\begin{enumerate}
    \item[\rm (i)] $\mathcal G$ is non-abelian, locally compact, totally disconnected, non-discrete and topologically simple.

    \item[\rm (ii)] $\mathcal G$ is not abstractly simple; every non-trivial normal subgroup of $\mathcal G$ contains $\Inn(G_\infty)$. 

    \item[\rm (iii)] The closure of the image of $Z(G_1)$ in $\mathcal G$ is a compact open subgroup isomorphic to  $\ZZ_{p_1}^d \times \dots \times \ZZ_{p_t}^d$, where $p_1 < p_2 < \dots < p_t$ denote  the distinct prime divisors of $p$. 

\item[\rm (iv)] $R$ is embedded as a discrete subgroup in $\mathcal G$.

\item [\rm (v)] $\mathcal G$ is second countable, but not compactly generated.

\item [\rm (vi)] For each $n$, the closure of the image of $Z(G_n)$ in $\mathcal G$ is open, and coincides with $C_\mathcal{G}(G_n)$. 

\item[\rm (vii)] For each $m \geq 1$ and each Hausdorff topological field $F$, the only continuous homomorphism $\mathcal G \to \mathrm{GL}_m(F)$ is the trivial one.

\item[\rm (viii)] Every finitely generated centerless subgroup of $\mathcal G$ is discrete. 

\item[\rm (ix)] The  finitely generated centerless subgroups of $\mathcal G$ fall into countably many abstract isomorphism classes.
\end{enumerate}
\end{thm}
\begin{proof}
(i) The group $\mathcal G$ is totally disconnected by \cite[Lemma~I.5~(iii)]{CM11}. Let $F_1 \subseteq G_1$ be a finite generating set of $G_1$. By Proposition~\ref{prop:asc-chain}~(iv), we have $C_{G_n}(G_1) = Z(G_1)$ for all     $n\in \N$, so that $C_{G_\infty}(F_1) = C_{G_\infty}(G_1) = Z(G_1)$. By Proposition~\ref{prop:asc-chain}~(iii), for each $n \in \N$ the group $Z(G_1)$ is  virtually a central subgroup of $G_n$, so that the conjugation action of $Z(G_1)$  on $G_n$ has finite orbits. Hence $\mathcal G$ is locally compact by Lemma~\ref{lem:Linn}~(ii). 

Since $G_\infty$ is infinite and simple, it has a trivial center. Therefore $G_\infty \cong \Inn(G_\infty) \leqslant \mathcal{G}$, so $\mathcal{G}$ is non-abelian.
On the other hand, we have $C_{G_\infty}(G_n) \geqslant Z(G_n) \cong \ZZ^d$, for all $n \in \N$, by Proposition~\ref{prop:asc-chain}~(ii), so that the centralizer of every finite subset of $G_\infty$ is non-trivial. Therefore $\mathcal G$ is non-discrete by Lemma~\ref{lem:Linn}~(iii).

The fact that $\mathcal G$ is topologically simple follows from Lemma~\ref{lem:Linn}~(iv), and will actually be recovered in the following paragraph. 

(ii) We have $\Inn(G_\infty) \lhd \mathcal G \lhd \Aut(G_\infty)$, so that $\Inn(G_\infty)$ is a dense normal subgroup of $\mathcal G$. The group $G_\infty$ is countable while $\mathcal G$ is a non-discrete locally compact group, hence it is uncountable by Baire's theorem. Therefore $\Inn(G_\infty)$ is strictly contained in $\mathcal G$. Since $G_\infty$ is simple, any non-trivial normal subgroup $N$ of  $\mathcal G$ either contains $\Inn(G_\infty)$ or intersects it trivially. The latter case does not occur since it would imply that $N$ centralizes $\Inn(G_\infty)$, whereas $N$ acts faithfully on $G_\infty$ since $N \leqslant \Aut(G_\infty)$.

(iii) 
By definition of the topology on $\mathcal G$, open neighborhoods of the identity are the stabilizers  of finite subsets of $G_\infty$. Since $G_\infty$ is the ascending union of the finitely generated groups $G_n$, we can conclude that the descending chain $\{V_n=\mathrm{Fix}_{\mathcal G}(G_n)\}_{n \in \N}$ is a basis of identity neighborhoods in $\mathcal G$.

Recall that $Z(G_\infty)=\{1\}$, so the action of $\mathcal G$ on $G_\infty$ is equivalent to the action of $\mathcal G$ by conjugation on $\Inn(G_\infty) \cong G_\infty$. To simplify the notation in the remainder of the proof we will identify $G_\infty$, $G_n$ and $Z(G_n)$ with their images in $\mathcal{G}$, for all $n \in \N$. In this new notation, we have $V_n=C_{\mathcal G}(G_n)$, and this is an open  subgroup of $\mathcal G$, for each $n \in \N$. Since $G_\infty$ is dense in $\mathcal G$, the intersection $V_n \cap G_\infty=C_{G_\infty}(G_n)$ is dense in $V_n$, i.e., 
\[\overline{C_{G_\infty}(G_n)}=V_n.\] 
Note that, by Proposition~\ref{prop:asc-chain}~(iv), 
\[C_{G_\infty}(G_n) \subseteq C_{G_\infty}(G_1) \subseteq C_{G_1}(G_1)=Z(G_1) \subseteq G_1 \subseteq G_n,\]
so $C_{G_\infty}(G_n)=C_{G_n}(G_n)=Z(G_n)$. Consequently, we obtain
\begin{equation}
\label{eq:struct_of_Vn}
V_n=C_{\mathcal G}(G_n)=\overline{Z(G_n)}, \text{ for all } n \in \N.
\end{equation}
Clearly the conjugation action of $Z(G_n)$ on $G_\infty$ has finite orbits, therefore $V_n$ is compact by \eqref{eq:struct_of_Vn} and Lemma~\ref{lem:Linn}~(i). Since $Z(G_n) \cong \Z^d$ is abelian and $\mathcal G$ is totally disconnected, we can conclude that $V_n=\overline{Z(G_n)}$ is an abelian profinite group.

Note that $\{V_n\}_{n \in \N}$ is also a basis of open neighborhoods in $V_1$, whence $V_1$ is the projective limit of the finite quotients  $\{V_1/V_n\}_{n \in \N}$,  see \cite[Theorem 1.2.5~(a)]{Wilson}. We shall now show that $V_1/V_n$ is isomorphic to $(\Z/p^{n-1}\Z)^d$.

By construction $V_n$ commutes with $G_n$, hence with $G_1$. Therefore  $Z(G_1)V_n$ is a subgroup of $\mathcal G$. Since $V_n$ is open, the subgroup $Z(G_1)V_n$ is open, hence closed. It follows that 
\[\overline{Z(G_1)} = V_1 \subseteq Z(G_1)V_n \subseteq V_1,\]
which shows that
$V_1 =  Z(G_1)V_n$, whence
\[V_1/V_n \cong Z(G_1)/(Z(G_1) \cap V_n).\]
On the other hand, we have $Z(G_1) \cap V_n = C_{Z(G_1)}(G_n) \subseteq Z(G_n)$,
so  $Z(G_1) \cap V_n = Z(G_n)$.
Therefore we conclude that
\[V_1/V_n \cong Z(G_1)/Z(G_n) \cong (\Z/p^{n-1}\Z)^d,\]
as required. It also follows that the map from  
$V_1/V_{n+1}$ to $V_1/V_n$ is induced by natural map from $Z(G_1)/Z(G_{n+1})$ to $Z(G_1)/Z(G_n)$, for all $n \in \N$. Therefore \[V_1 \cong \varprojlim V_1/V_n \cong  \varprojlim Z(G_1)/Z(G_n).\] 

Now, for each $n \in \N$ the group $Z(G_1)/Z(G_n) \cong (\Z/p^{n-1}\Z)^d$ is isomorphic to the direct sum $\displaystyle \mathop{\oplus}_{i=1}^t\left(\ZZ/p_i^{e_i(n-1)}\ZZ\right)^d$, where
$p_1,\dots,p_t$ is the list of distinct prime factors of $p$ and $e_i \geq 1$ is the largest integer such that $p_i^{e_i}$ divides $p$. In view of
Proposition~\ref{prop:asc-chain}~(iii) it is now easy to see that 
\[V_1 \cong \varprojlim Z(G_1)/Z(G_n) \cong \ZZ_{p_1}^d \times \dots \times \ZZ_{p_t}^d.\]

(iv) According to Proposition~\ref{prop:asc-chain}~(v), $R$ is subgroup of $G_1$ such that $R \cap Z(G_1)=\{1\}$. Therefore \[R \cap V_1=R \cap C_{\mathcal G}(G_1)=R \cap Z(G_1)=\{1\},\]
which immediately implies that $R$ is a discrete subgroup of $\mathcal G$, because $V_1$ is an open neighborhood of the identity.

(v) and (vi) In view of (iii), the profinite group $V_1$ is metrizable,  hence $\mathcal G$ is first countable, which then implies that $\mathcal G$ is metrizable (see \cite[Theorem~II.8.3]{HW}). By construction, $\mathcal G$ contains the countable group $G_\infty$ as a dense subgroup, hence $\mathcal G$ is also separable. This confirms that $\mathcal G$ is second countable. 

Recall that $V_n=C_\mathcal{G}(G_n)$ are compact open subgroups of $\mathcal G$. According to equation \eqref{eq:struct_of_Vn}, $V_n = \overline{Z(G_n)} \leqslant \overline{G_n}$, so the groups $\overline{G_n}$ form an ascending chain of (compactly generated)
open subgroups of $\mathcal G$, with $\mathcal G = \bigcup_n \overline{G_n}$. 
It follows  that $\overline{G_n} = G_n V_n$, hence 
\begin{equation}\label{eq:oG_n/V_n}
 \overline{G_n}/V_n \cong G_n/(G_n \cap V_n) = G_n/(G_n \cap C_\mathcal{G}(G_n)) =G_n/Z(G_n). 
\end{equation}
 Thus $\overline{G_n}/V_n$ is a non-trivial simple group, which implies that 
the group $\overline{G_n}$ is not topologically simple (indeed, it is non-discrete with an open center), so it must be a proper subgroup of $\mathcal G$ by (i). 

Since the subgroups $\{\overline{G_n}\}_{n \in \N}$ form a nested open cover of $\mathcal G$, every compactly generated subgroup of $\mathcal G$ is contained in $\overline{G_n}$ for some $n$. This  indeed shows that $\mathcal G$ is not compactly generated.

(vii) If such a non-trivial continuous homomorphism existed, it would be injective since $\mathcal G$ is topologically simple. In particular, for each $n$, the group $G_n$ would be linear since it embeds as a subgroup of $\mathcal G$. This is impossible, since a finitely generated linear group is residually finite, whereas $G_n$ does not have non-trivial finite quotients by Proposition~\ref{prop:asc-chain} (indeed, any finite quotient of $G_n$ must be abelian as $G_n/Z(G_n)$ is infinite simple, so it is trivial because $G_n$ is perfect).

(viii) and (ix) Let $H$ be a finitely generated subgroup of $\mathcal{G}$ with $Z(H)=\{1\}$. Since $\mathcal{G}$ is the ascending union $\bigcup_{m \in \N} \overline{G_m}$, there must exist $n \in \N$ such that $H \leqslant \overline{G_n}$.

Note that, by \eqref{eq:struct_of_Vn}, $V_n=\overline{Z(G_n)}=C_\mathcal{G}(G_n) = Z(\overline{G_n})$, so $H \cap V_n=\{1\}$ as $H$ is centerless. Since $V_n$ is an open subgroup of $\mathcal G$, we can conclude that $H$ must be discrete. In view of \eqref{eq:oG_n/V_n}, we also have that $H$ embeds into the finitely generated group $\overline{G_n}/V_n \cong G_n/Z(G_n)$.

Now, the claim follows from the observation that the countable family of finitely generated groups $\{G_n/Z(G_n)\}_{n \in \N}$ has only countably many finitely generated subgroups.
\end{proof}

\begin{rem}
Recall from \cite{Wil07} (see also \cite[Theorem~A]{CRW} for a more general statement) that a compactly generated topologically simple totally disconnected locally compact group cannot have a solvable open subgroup. In particular, the fact that the group $\mathcal G$ from Theorem~\ref{thm:technical-version} is not compactly generated directly follows. We have preferred to include a direct argument, which conveys more information on the structure of $\mathcal G$. 
\end{rem}

\begin{proof}[Proof of Theorem \ref{thm:main}]
Let $p$ be a prime and $d \geq 1$ be an integer. Given a countable  group $R$, let $\mathcal G(R)$ be  the  second countable, topologically simple, locally compact group that contains $R$ as a discrete subgroup, as provided by Theorem~\ref{thm:technical-version}. The theorem ensures that $\mathcal G(R)$ is locally isomorphic to the direct sum $\ZZ_p^d$ of $d$-copies of the additive group of $p$-adic integers. Therefore $\mathcal G(R)$ is a $p$-adic analytic group whose $\QQ_p$-Lie algebra is $\QQ_p^d$ (see \cite[Corollary~8.33]{DDMS}).

Since there exist continuously many isomorphism classes of finitely generated infinite simple (hence, centerless)  groups (see \cite[Corollary]{Camm}), we deduce from Theorem~\ref{thm:technical-version}~(ix) that the family of simple $p$-adic groups $\{\mathcal G(R) \mid R \text{ is finitely generated and centerless}\}$ falls into $2^{\aleph_0}$ abstract isomorphism classes.
\end{proof}

\begin{rem}
Recall that every second countable locally compact Hausdorff space is Polish. In particular, so is every second countable, locally compact group. Given a Polish group $G$, let $S$ be a (countable) subgroup generated by a countable dense subset of $G$; we fix a complete metrization of $G$ and think of $S$ as a metric space. Note that $G$ is the metric completion of $S$ and the multiplication on $G$ is uniquely defined by its restriction to $S$, by continuity. Therefore, the topological group isomorphism type of $G$ is uniquely defined by the algebraic isomorphism type and the metric structure of $S$. Since there are at most $2^{\aleph_0}$ group structures and metrics on a countable set, there are at most $2^{\aleph_0}$ second countable, locally compact groups up to topological group isomorphism. Thus, we cannot have more than $2^{\aleph_0}$ isomorphism classes in Theorem \ref{thm:main}.
\end{rem}

\begin{proof}[Proof of Corollary \ref{cor:extraordinary}]
Let $\mathcal G$ be a group provided by Theorem~\ref{thm:technical-version}. Clearly, $\mathcal G$ satisfies parts (a) and (b) in the definition of an extraordinary group. Since every open subgroup of a topological group is also closed, the only open subnormal subgroup of $\mathcal G$ is $\mathcal G$ itself. 

Arguing by contradiction, suppose that $[\mathcal G,\mathcal G]$ is discrete. A discrete subgroup of a Hausdorff  topological group is necessarily closed, hence  $[\mathcal G,\mathcal G]$ is a closed normal subgroup of $\mathcal G$. Since $\mathcal G$ is not abelian, by topological simplicity we conclude that $[\mathcal G,\mathcal G]=\mathcal G$. The latter implies that $\mathcal G$ is discrete, contradicting Theorem~\ref{thm:technical-version}. Thus $[\mathcal G,\mathcal G]$ cannot be discrete.
\end{proof}

\subsection{A remark on distality}  Recall that a locally compact group $\mathcal G$ is called \emph{distal} if the conjugation action of $\mathcal G$ on itself is distal. This can be characterized by the property that for each non-trivial $g \in \mathcal G$, the closure of the conjugacy class of $g$ does not contain the neutral element. For background and various results on distal groups, we refer to \cite{Rei20}.

Since every topologically simple group is unimodular, the following observation shows that the $1$-dimensional groups from Theorem~\ref{thm:technical-version} (where $d=1$) are all distal. In fact, one can show that  the topologically simple groups in Theorem~\ref{thm:technical-version} are also all distal. We omit the details.

\begin{lem}\label{lem:distal}
    Let $\mathcal G$ be a  locally compact group with a compact open subgroup $V$ isomorphic to $\ZZ_p$ for some prime $p$. If $\mathcal G$ is unimodular, then it is distal. 
\end{lem}
\begin{proof}
    Let $v \in \mathcal G$ be a non-trivial element and $(g_k) \subset \mathcal G$ be a sequence such that the sequence $(g_k v g_k^{-1})$ converges to the neutral element. (Since $G$ is locally isomorphic to $\ZZ_p$, it is first countable, and we may thus use sequences.)  
Upon  replacing $v$ by $g_m v g_{m}^{-1}$ for a suitable $m$ and $(g_k)$ by $(g_k g_m^{-1})$, we may assume, without loss of generality, that $v \in V$. Let $\varphi \colon V \to \ZZ_p$ be the isomorphism given by the hypothesis. Since every non-trivial closed subgroup of $\ZZ_p$ is open, it follows that $[V : \overline{\la v \ra}] = p^{n_0}$, for some $n_0 \ge 0$. For all $m \in \N$, there exists a  sufficiently large $k \in \N$ such that  $g_k v g_k^{-1}$ belongs to $\varphi^{-1}(p^m \ZZ_p)$. It follows that $[V : \overline{\la g_k v g_k^{-1} \ra}] \geq  p^{m}$. For $m >n_0$, it follows that the respective measures of the compact open subgroups $\overline{\la v \ra}$ and $\overline{\la g_k v g_k^{-1} \ra}$ are different. In particular, the Haar measure on $\mathcal G$ is not invariant under conjugation, so $\mathcal G$ is not unimodular.
\end{proof}

\end{document}